\documentclass[11pt]{article}

\usepackage[margin=1in]{geometry}
\usepackage[T1]{fontenc}
\usepackage[utf8]{inputenc}
\usepackage{lmodern}
\usepackage{microtype}
\usepackage{amsmath,amssymb,amsthm,mathtools}
\usepackage{booktabs,array}
\usepackage{graphicx}
\usepackage{enumitem}
\usepackage[hidelinks]{hyperref}
\usepackage{doi}
\usepackage{authblk}
\usepackage[backend=biber,citestyle=numeric-comp,bibstyle=ieee,sorting=none,minbibnames=5,maxbibnames=8,giveninits=true]{biblatex}
\AtBeginBibliography{\small}
\hypersetup{
  pdftitle={Toeplitz multiplication and graded factorization of determinant recurrences},
  pdfauthor={Max A. Alekseyev and Dmitry I. Khomovsky}
}

\newtheorem{theorem}{Theorem}[section]
\newtheorem{proposition}[theorem]{Proposition}
\newtheorem{lemma}[theorem]{Lemma}
\newtheorem{corollary}[theorem]{Corollary}

\newtheorem{example}[theorem]{Example}
\newtheorem{definition}[theorem]{Definition}
\theoremstyle{remark}
\newtheorem{remark}[theorem]{Remark}

\DeclareMathOperator{\rank}{rank}
\newcommand{\T}{T}
\newcommand{\cR}{\mathcal R}
\newcommand{\cK}{\mathcal K}
\newcommand{\cP}{\mathcal P}
\newcommand{\defeq}{:=}

\title{Toeplitz multiplication and graded factorization of determinant recurrences}
\author[1]{Max A. Alekseyev}
\author[2]{Dmitry I. Khomovsky}
\affil[1]{\small The George Washington University, Washington, DC, USA. Email: \href{mailto:maxal@gwu.edu}{maxal@gwu.edu}}
\affil[2]{\small Email: \href{mailto:khomovskij@physics.msu.ru}{khomovskij@physics.msu.ru}}
\date{}

\begin{document}
\maketitle

\begin{abstract}
We study how determinant recurrences of banded Toeplitz matrices behave when
their Laurent symbols are multiplied.  Two classical structures underlie the
problem.  Clean banded Toeplitz determinants have a root-product description of
their recurrences, while a product of two finite Toeplitz sections differs from
the finite section of the product by finite-rank corner corrections.  We
connect these structures by identifying multiplication-induced boundary defects
with exterior-degree sectors of the determinant recurrence.

To each polynomial core we associate recurrence polynomials for all exterior
degrees.  Multiplication of cores becomes a graded convolution via composed
products, while complementary degrees are related by a scaled reciprocal
duality.  On the matrix side, one-sided factors produce a defect filtration
whose successive pieces are exactly these sectors.  For general two-sided
factors, independently weighting the two corner corrections gives an interval
filtration: increasing either defect order adds one adjacent sector, and the
cumulative recurrences are generically minimal.  The multiplication-induced
same-corner minors used in these filtrations lie in the standard row-column
state module, whereas cross-corner imbalance corresponds instead to Laurent
recentering.

Opposite Laurent recenterings realize the individual sectors directly, and for
several factors the admissible recenterings form an explicit lattice polytope.
The pentadiagonal case gives the basic $1+4+1$ decomposition.  The resulting
framework gives a factor-level description of how finite-section boundary
effects generate the recurrence spectrum of structured banded Toeplitz
products.
\end{abstract}

\noindent\textbf{Key words.} banded Toeplitz matrix; determinant recurrence; finite-section product; finite-rank correction; exterior power; composed product.

\medskip
\noindent\textbf{MSC codes.} 15B05; 15A15; 15A75.

\section{Introduction}

Let
\[
 a(z)\defeq\sum_{\nu=-r_a}^{s_a}a_\nu z^\nu,
 \qquad a_{-r_a}a_{s_a}\ne0,
\]
be a banded Laurent polynomial and let
\[
 \T_n(a)\defeq(a_{j-i})_{i,j=1}^n
\]
be the $n\times n$ leading finite section of its bi-infinite Toeplitz matrix.
The interval $[-r_a,s_a]$ is the support of $a$; $r_a$ and $s_a$ are its lower
and upper semibandwidths.  We call
\[
 q_a(z)\defeq z^{r_a}a(z)
\]
the \emph{polynomial core}.  The determinant sequence
$D_n(a)\defeq\det\T_n(a)$ satisfies a constant-coefficient recurrence.  In the
generic case its minimal recurrence order is
$\binom{r_a+s_a}{s_a}$, and the recurrence modes are fixed-cardinality
products of the roots of $q_a$; see the companion paper
\cite{AlekseyevKhomovskyRecurrences2026} and the classical root-product formula
of Widom \cite{Widom1958}.  Proposition~\ref{prop:background-recurrence}
restates the precise recurrence input used below.  Here and throughout,
\emph{generic} means outside a proper algebraic exceptional set of coefficient
values.

We ask how this recurrence structure behaves under multiplication.  For a
second banded symbol $b$ put $c\defeq ab$.  On bi-infinite Toeplitz matrices
multiplication is exact, whereas for finite sections
\[
 \T_n(a)\T_n(b)
\]
differs from $\T_n(c)$ only near two opposite corners.  These finite-rank
corrections are the \emph{corner defects}.  At the same time,
\[
 \det(\T_n(a)\T_n(b))=D_n(a)D_n(b),
\]
so a factorization of the symbol gives immediate recurrence information about
an almost-Toeplitz product.  The issue is to identify the additional recurrence
modes supplied when the corner defects are restored.

Already two tridiagonal factors show the mechanism.  Their determinant
sequences generically have order two, so their termwise product has at most
four modes, whereas a generic pentadiagonal Toeplitz determinant has order six.
For coprime quadratic cores, the companion space splits as
\[
 \bigwedge^2(V_a\oplus V_b)
 \simeq
 \bigwedge^2V_a\oplus(V_a\otimes V_b)\oplus\bigwedge^2V_b,
\]
with dimensions $1+4+1$.  The middle summand is seen by the product of the two
order-two determinant recurrences; the two one-dimensional summands are the
missing modes.  More generally, we call the summands obtained by distributing
an exterior degree among the factor spaces \emph{sectors}.

This leads to two compatible descriptions.  On the recurrence side, we attach
to a polynomial core its full graded family of exterior-power recurrence
polynomials.  Multiplication of cores becomes a graded convolution, with
pairwise products of modes encoded by the classical composed product.  On the
finite-section side, opposite Laurent shifts
\[
 c=(z^\delta a)(z^{-\delta}b)
\]
preserve the product symbol while redistributing corner-defect rank.  We call
this operation \emph{recentring}, and $\delta$ the \emph{defect index}.  The
admissible defect indices are in bijection with the graded sectors.

\subsection{Relation to previous work and contributions}\label{sec:related-work}
Widom's formula gives the fixed-cardinality root products governing banded
Toeplitz determinants \cite{Widom1958}; Alexandersson placed these roots in a
Schur-polynomial recurrence framework \cite{Alexandersson2012}, with related
skew-Schur formulas for banded Toeplitz minors due to Maximenko and
Moctezuma-Salazar \cite{MaximenkoMoctezuma2017}.  The finite-section product
identity for Toeplitz operators is also classical: the discrepancy between the
Toeplitz section of a product and the product of the two sections is expressed
through Hankel boundary terms \cite{Widom1976}; see also
\cite{BottcherSilbermann2006,Virtanen2022}.  Finite-dimensional algebras
associated with Toeplitz matrix sequences were studied, for example, by
Serra-Capizzano \cite{SerraCapizzano2001}.  Toeplitz products and
Toeplitz-plus-low-rank representations also occur in structured numerical
linear algebra: Chandrasekaran and Sayed treat products of Toeplitz matrices as
shift-structured systems \cite{ChandrasekaranSayed1998}, while quasi-Toeplitz
arithmetic exploits Toeplitz-plus-low-rank representations for finite and
semi-infinite computations \cite{BiniMasseiRobol2019}.  Thus neither
root-product modes nor the finite-section product mechanism is new here.  The
novel contribution of this paper is their exact coupling: finite-section defect
rank and defect order are identified with exterior-degree recurrence sectors,
leading to a multiplicative graded recurrence calculus, generically sharp
defect filtrations, and a recentering description of the individual sectors.
Our companion preprint \cite{AlekseyevKhomovskyRecurrences2026} develops recurrence
constructions for a single band and proves generic minimality.  We use that
result here as recurrence input; the multiplication law, defect filtrations,
and recentering theory developed below are separate from that construction.

Our first result is a multiplicative calculus for the graded recurrence
profile.  The profile of a product core is the graded convolution of the factor
profiles, and complementary exterior degrees are related by a scaled reciprocal
duality.  For finite sections we determine the exact ranks of the two corner
defects.  For one-sided products, interpolation from the exact finite product
to the clean Toeplitz section produces a defect-order filtration whose
cumulative annihilator through order $k$ is the product of the first $k+1$
sector polynomials and is generically minimal.  The filtration is precisely the
standard filtration of $\bigwedge^s(V_a\oplus V_b)$ by $V_a$-exterior degree.
For two-sided products, independently weighting the two corner defects gives
an interval bifiltration: bidegree $(k,\ell)$ contains exactly the sectors with
defect indices $-\ell,\ldots,k$, generically, and the corresponding cumulative
annihilator is generically minimal.

Recentring realizes the individual sectors directly and transfers defect rank
between the two boundaries.  For several factors, the cumulative defect
indices form a bounded lattice polytope whose integer points are exactly the
multi-index sectors.  Multiplication by a bidiagonal factor gives the basic
rank-one/Pascal specialization, while two tridiagonal factors recover the
classical degree-six pentadiagonal recurrence as the $1+4+1$ decomposition.
The five-diagonal literature includes explicit determinant formulas and
generating functions \cite{MarrVineyard1988,AndelicDaFonseca2021},
tridiagonal factorization methods \cite{DieleLopez1998}, and homogeneous
recurrence formulas \cite{JiaYangLi2016,Sweet}; our use of this case is to
expose its graded sector structure.

From the matrix-analysis viewpoint, the results identify exactly which
recurrence modes are generated by finite-rank boundary corrections under a
structured Toeplitz factorization.  They also provide factor-by-factor,
root-free formulas for assembling the canonical determinant recurrence when a
factorization of the polynomial core is available.  This is potentially useful
for symbolic and structured determinant computations and for analyzing how
finite-section products alter recurrence spectra.  The present results are
exact algebraic statements; we make no numerical-stability or performance
claims.

The paper is organized as follows.  Section~\ref{sec:profiles} develops graded
recurrence profiles, complementary-sector duality, and the multiplication law.
Section~\ref{sec:finite-products} proves the finite-section defect formula and
the one-sided defect-order filtration.
Section~\ref{sec:one-diagonal} treats one-diagonal growth and its Pascal rule.
Section~\ref{sec:defect-indices} identifies recentering sectors and proves the
two-corner interval filtration.  Section~\ref{sec:pentadiagonal} gives the
pentadiagonal $1+4+1$ prototype.  Section~\ref{sec:multifactor} treats arbitrary
products and the defect-index polytope, and Section~\ref{sec:algebra} records
the resulting recurrence-profile algebra.  Appendix~\ref{app:boundary-closure}
uses the companion row-column machinery to eliminate the standard-sector,
same-corner minors produced by finite-section multiplication, explains why
cross-corner imbalance instead changes the Laurent profile, and records the
exceptional geometries that collapse to one or two clean determinant terms.
Appendix~\ref{app:factorized-construction} gives the output-sensitive
construction for factorized polynomial cores.

\section{Graded recurrence profiles}\label{sec:profiles}

We work over a field $K$ and, when needed, pass to an algebraic closure.
Consider
\[
 q(z)\defeq q_0+q_1z+\cdots+q_dz^d,
 \qquad q_0q_d\ne0.
\]
For $0\le k\le d$, set
\begin{equation}\label{eq:standard-symbol}
 a_{q,k}(z)\defeq z^{k-d}q(z).
\end{equation}
This Laurent polynomial has lower and upper semibandwidths $d-k$ and $k$,
respectively.  We call $k$ the \emph{exterior degree}, because the companion-
matrix construction uses the $k$th exterior power $\bigwedge^k$ to obtain the
corresponding recurrence.  Thus moving the main diagonal through the same
polynomial core $q$ produces a family of Toeplitz determinant problems indexed
by $k$.

Let $\rho_1,\ldots,\rho_d$ denote the roots of $q$, counted with
multiplicity, and write
\[
 [d]\defeq\{1,\ldots,d\}.
\]
The \emph{degree-$k$ recurrence polynomial} associated with $q$ is
\begin{equation}\label{eq:chi-qk}
 \chi_{q,k}(t)
 \defeq
 \prod_{\substack{I\subseteq[d]\\ |I|=k}}
 \left(
 t-(-1)^kq_d\prod_{i\in I}\rho_i
 \right).
\end{equation}
Its roots are the fixed-cardinality products that occur in Widom's determinant
formula and, more generally, in recurrence formulas for banded Toeplitz minors
\cite{Widom1958,Alexandersson2012}.  We refer to these roots as the
\emph{degree-$k$ recurrence modes}.  The coefficients of \eqref{eq:chi-qk}
are symmetric in the roots and therefore lie in the coefficient field of $q$.

Let $C_q$ denote the Frobenius companion matrix of the monic polynomial
$q/q_d$.  The induced linear map on the $k$th exterior power is denoted by
$\bigwedge^k C_q$; in matrix terminology it is the $k$th compound matrix of
$C_q$.

\begin{proposition}[Toeplitz determinant recurrence]\label{prop:background-recurrence}
For every $0\le k\le d$, the polynomial $\chi_{q,k}$ annihilates the
sequence
\[
 D_n(q,k)\defeq\det\T_n(a_{q,k}).
\]
Equivalently, $\chi_{q,k}$ is the characteristic polynomial of the normalized
compound matrix
\[
 (-1)^kq_d\,\bigwedge^k C_q.
\]
For generic coefficients of $q$, $\chi_{q,k}$ is the minimal recurrence
polynomial of $D_n(q,k)$, and the generic minimal recurrence order is
$\binom{d}{k}$.
\end{proposition}

\begin{proof}
This is the determinant specialization of the exterior-power construction in
\cite{AlekseyevKhomovskyRecurrences2026}; the explicit eigenvalues of the
normalized compound matrix are precisely the factors displayed in
\eqref{eq:chi-qk}.
\end{proof}

\begin{remark}[Nongeneric specializations]\label{rem:nongeneric}
Proposition~\ref{prop:background-recurrence} has a canonical and a minimal
part.  The polynomial $\chi_{q,k}$ remains a valid annihilator after every
specialization of the coefficients for which $q_0q_d\ne0$.  At exceptional
parameter values, however, the minimal recurrence polynomial may be a proper
divisor of $\chi_{q,k}$.  This can happen because distinct subset-product
modes coincide, because some modes have zero coefficient in the scalar
determinant sequence, or because a repeated root of $q$ produces a confluent
situation.  In particular, a collision of two subset products alone does not
determine the multiplicity of the corresponding root in the minimal recurrence
polynomial.  All profile identities below concern the canonical annihilator
$\chi_{q,k}$; they do not attempt to remove factors that become redundant
only after a nongeneric specialization.
\end{remark}

The extreme exterior degrees are one-dimensional and satisfy
\begin{equation}\label{eq:extreme-profiles}
 \chi_{q,0}(t)=t-q_d,
 \qquad
 \chi_{q,d}(t)=t-q_0.
\end{equation}
Indeed, the product of all roots is $(-1)^dq_0/q_d$.

\subsection{Complementary-sector duality}\label{subsec:duality}

The graded profile has an intrinsic redundancy: exterior degrees $k$ and
$d-k$ determine one another.  Set
\begin{equation}\label{eq:duality-data}
 N_k\defeq\binom dk,
 \qquad
 A_k\defeq\binom{d-1}{k-1},
 \qquad
 B_k\defeq\binom{d-1}{k},
 \qquad
 K_q\defeq q_0q_d,
\end{equation}
where a binomial coefficient outside its usual range is interpreted as zero.

\begin{theorem}[Complementary-sector duality]\label{thm:sector-duality}
For every $0\le k\le d$,
\begin{equation}\label{eq:sector-duality}
 \chi_{q,d-k}(t)
 =
 (-1)^{N_k}q_0^{-A_k}q_d^{-B_k}\,
 t^{N_k}\chi_{q,k}\!\left(\frac{K_q}{t}\right).
\end{equation}
Thus the canonical annihilators in exterior degrees $k$ and $d-k$ are scaled
reciprocals of one another.  The identity remains valid under nongeneric
specializations for which $q_0q_d\ne0$.
\end{theorem}

\begin{proof}
Write the roots of $q$ as $\rho_1,\ldots,\rho_d$.  For a $k$-element subset
$I\subseteq[d]$, write
\[
 W_I^{(k)}\defeq(-1)^kq_d\prod_{i\in I}\rho_i
\]
for the corresponding root of $\chi_{q,k}$.  Its complementary subset $I^c$
has size $d-k$, and
\begin{align}
 W_I^{(k)}W_{I^c}^{(d-k)}
 &=(-1)^dq_d^2\prod_{i=1}^{d}\rho_i\notag\\
 &=q_0q_d
 =K_q.
 \label{eq:complementary-modes}
\end{align}
Hence complementation sends the root multiset of $\chi_{q,k}$ to the root
multiset obtained by $W\mapsto K_q/W$.

It remains to determine the monic normalization.  The product of all roots of
$\chi_{q,k}$ is
\begin{equation}\label{eq:product-sector-roots}
 \prod_{|I|=k}W_I^{(k)}=q_0^{A_k}q_d^{B_k}.
\end{equation}
Indeed, each $\rho_i$ occurs in $A_k$ subsets, while
$N_k-A_k=B_k$; the sign is positive because
$kN_k=dA_k$.  The standard reciprocal-polynomial identity now gives
\eqref{eq:sector-duality}.
\end{proof}

The self-dual case is especially relevant for balanced Toeplitz bands.

\begin{corollary}[Balanced self-reciprocity]\label{cor:balanced-self-duality}
Let $d=2m$, $k=m$, and
\begin{equation}\label{eq:balanced-duality-data}
 N\defeq\binom{2m}{m},
 \qquad
 K_q\defeq q_0q_{2m}.
\end{equation}
Then $N$ is even and
\begin{equation}\label{eq:balanced-self-reciprocal}
 \chi_{q,m}(t)
 =K_q^{-N/2}t^N\chi_{q,m}\!\left(\frac{K_q}{t}\right).
\end{equation}
Writing
\begin{equation}\label{eq:balanced-coefficients}
 \chi_{q,m}(t)\defeq\sum_{j=0}^{N}c_jt^j,
 \qquad c_N\defeq1,
\end{equation}
we have
\begin{equation}\label{eq:balanced-coefficient-symmetry}
 c_j=K_q^{N/2-j}c_{N-j},
 \qquad 0\le j\le N.
\end{equation}
\end{corollary}

\begin{proof}
For $d=2m$ and $k=m$,
\[
 A_m=B_m=\frac N2,
\]
and $N=2\binom{2m-1}{m-1}$ is even.  Substitution into
Theorem~\ref{thm:sector-duality} gives
\eqref{eq:balanced-self-reciprocal}; comparing coefficients gives
\eqref{eq:balanced-coefficient-symmetry}.
\end{proof}

\begin{corollary}[Half-profile reconstruction]\label{cor:half-profile}
The full graded profile $\cR(q)$ is determined by the components
\[
 \chi_{q,k},
 \qquad
 0\le k\le\left\lfloor\frac d2\right\rfloor.
\]
The degree-weighted size of these independently stored components is
\begin{equation}\label{eq:half-profile-volume}
 H_d\defeq
 \sum_{k=0}^{\lfloor d/2\rfloor}\binom dk
 =
 \begin{cases}
  2^{d-1}, & d\text{ odd},\\[1ex]
  2^{d-1}+\dfrac12\binom d{d/2}, & d\text{ even}.
 \end{cases}
\end{equation}
In particular, $H_d/2^d\to1/2$ as $d\to\infty$.
\end{corollary}

\begin{proof}
Theorem~\ref{thm:sector-duality} reconstructs every component above the middle
from its complementary component below the middle by coefficient reversal and
scaling.  Formula~\eqref{eq:half-profile-volume} is the usual symmetry of the
binomial coefficients about $d/2$.
\end{proof}

\subsection{Composed products}

For monic polynomials $f$ and $g$ of degrees $m$ and $n$, with roots
$\alpha_1,\ldots,\alpha_m$ and $\beta_1,\ldots,\beta_n$, their
\emph{composed product} is defined by
\begin{equation}\label{eq:boxtimes}
 (f\boxtimes g)(t)
 \defeq
 \prod_{i=1}^m\prod_{j=1}^n(t-\alpha_i\beta_j).
\end{equation}
This operation is associative and commutative at the level of root multisets.
It can be computed without adjoining the roots.  Let
$\operatorname{Res}_x(F,G)$ denote the polynomial resultant of $F$ and $G$
with respect to $x$.  Then
\begin{equation}\label{eq:resultant-boxtimes}
 (f\boxtimes g)(t)
 =\operatorname{Res}_x\!\left(f(x),x^n g(t/x)\right).
\end{equation}
The resultant formula is conceptual rather than the only way to compute the
operation.  Bostan, Flajolet, Salvy, and Schost give fast root-free algorithms
for composed products, including a characteristic-zero algorithm of
quasi-linear cost in the degree of the resulting polynomial
\cite{BostanFlajoletSalvySchost2006}.

For a product decomposition $q\defeq q_1q_2$ and a fixed total exterior degree $k$, the
\emph{$i$th graded sector} is the contribution obtained by choosing $i$ roots
from $q_1$ and $k-i$ roots from $q_2$.  The next theorem says
that the full recurrence polynomial is the product of the sector
polynomials.

\begin{theorem}[Graded multiplication law]\label{thm:graded-product}
Let $q_1,q_2\in K[z]$ have nonzero constant and leading coefficients.  Set
$d_1\defeq\deg q_1$, $d_2\defeq\deg q_2$, and $q\defeq q_1q_2$.  Then, for every
$0\le k\le d_1+d_2$,
\begin{equation}\label{eq:graded-product}
 \chi_{q,k}(t)
 =
 \prod_{i=\max(0,k-d_2)}^{\min(d_1,k)}
 \left(
   \chi_{q_1,i}\boxtimes\chi_{q_2,k-i}
 \right)(t).
\end{equation}
In particular,
\begin{equation}\label{eq:degree-vandermonde}
 \deg\chi_{q,k}
 =\sum_{i+j=k}\binom{d_1}{i}\binom{d_2}{j}
 =\binom{d_1+d_2}{k}.
\end{equation}
\end{theorem}

\begin{proof}
Write the roots of $q_1$ as $\alpha_1,\ldots,\alpha_{d_1}$ and those of
$q_2$ as $\beta_1,\ldots,\beta_{d_2}$.  A recurrence mode of
$\chi_{q,k}$ is obtained by choosing a $k$-element submultiset from the
disjointly labelled union of the $\alpha$- and $\beta$-roots.  Fix such a
choice containing $i$ roots from $q_1$ and $j$ roots from $q_2$, where
$j\defeq k-i$.  The corresponding mode is
\[
 (-1)^k(q_1)_{d_1}(q_2)_{d_2}
 \prod_{\alpha\in I}\alpha
 \prod_{\beta\in J}\beta.
\]
Since $k=i+j$, this factors as
\[
 \left((-1)^i(q_1)_{d_1}\prod_{\alpha\in I}\alpha\right)
 \left((-1)^j(q_2)_{d_2}\prod_{\beta\in J}\beta\right).
\]
Thus the sector with fixed $(i,j)$ has recurrence polynomial
$\chi_{q_1,i}\boxtimes\chi_{q_2,j}$.  The possible values of $i$ partition
all $k$-subsets, proving \eqref{eq:graded-product}.  Taking degrees gives
Vandermonde's identity \eqref{eq:degree-vandermonde}.
\end{proof}

\begin{remark}[Companion-matrix decomposition]\label{rem:transfer-decomp}
With $C_q$ as defined above, view it as a linear map on
$V_q\defeq K^d$.  For the two factors, write $V_1\defeq K^{d_1}$ and
$V_2\defeq K^{d_2}$.  When $q_1$ and $q_2$ are coprime, the Chinese remainder
theorem identifies the space for the product companion matrix with the direct
sum $V_1\oplus V_2$.  Hence
\[
 \bigwedge^k(V_1\oplus V_2)
 \simeq
 \bigoplus_{i+j=k}
 \bigl(\bigwedge^iV_1\otimes\bigwedge^jV_2\bigr).
\]
With the normalization
$(-1)^k q_d\bigwedge^kC_q$ from the companion paper, the scalar factor on
each summand is exactly the product of the scalar factors of the two smaller
induced matrices on the two exterior powers.  Thus
\eqref{eq:graded-product} is the recurrence-polynomial shadow of an
actual block decomposition.  Coprimality is needed
for this direct-sum similarity statement, but not for the root-multiset
identity \eqref{eq:graded-product}.
\end{remark}

\section{Products of finite Toeplitz sections}\label{sec:finite-products}

We now turn to the matrix side.  Consider
\[
 a(z)\defeq\sum_{\nu=-r_a}^{s_a}a_\nu z^\nu,
 \qquad
 b(z)\defeq\sum_{\nu=-r_b}^{s_b}b_\nu z^\nu,
\]
with support intervals $[-r_a,s_a]$ and $[-r_b,s_b]$, respectively; the
four endpoint coefficients are nonzero, so both support intervals are exact.
Set $c(z)\defeq a(z)b(z)$ and
\[
 A_n\defeq\T_n(a),\qquad B_n\defeq\T_n(b),\qquad C_n\defeq\T_n(c).
\]
We call $C_n-A_nB_n$ the \emph{finite-section defect}.  The calculation below
shows that it splits into two corrections supported at opposite corners.

\subsection{Exact corner-defect formula}

For $1\le i,j\le n$,
\begin{align}
 (C_n-A_nB_n)_{ij}
 &={}
 \sum_{k\le0}a_{k-i}b_{j-k}
 +\sum_{k\ge n+1}a_{k-i}b_{j-k}.
 \label{eq:missing-indices}
\end{align}
The first sum is supported in the first $r_a$ rows and first $s_b$ columns;
the second is supported in the last $s_a$ rows and last $r_b$ columns.

Set
\begin{equation}\label{eq:hLhR}
 h_L\defeq\min(r_a,s_b),
 \qquad
 h_R\defeq\min(s_a,r_b).
\end{equation}
For the left corner, set
\[
 (H_a^-)_{i\ell}\defeq a_{-(i+\ell-1)},
 \qquad
 (H_b^+)_{\ell j}\defeq b_{j+\ell-1}.
\]
Here $1\le i\le r_a$, $1\le j\le s_b$, and $1\le\ell\le h_L$;
entries outside the exact support are interpreted as zero.  Then the left
corner block is $H_a^-H_b^+$.  Similarly, after reversing the last rows and
columns, the right corner block is $H_a^+H_b^-$ with
\[
 (H_a^+)_{i\ell}\defeq a_{i+\ell-1},
 \qquad
 (H_b^-)_{\ell j}\defeq b_{-(j+\ell-1)},
\]
for $1\le i\le s_a$, $1\le j\le r_b$, and
$1\le\ell\le h_R$.

\begin{proposition}[Finite-section multiplication]\label{prop:finite-product}
For every $n$,
\begin{equation}\label{eq:finite-product}
 A_nB_n=C_n-L_n-R_n,
\end{equation}
where $L_n$ is supported in the upper-left $r_a\times s_b$ corner and $R_n$
in the lower-right $s_a\times r_b$ corner.  For all sufficiently large $n$,
\begin{equation}\label{eq:corner-ranks}
 \rank L_n=h_L,
 \qquad
 \rank R_n=h_R.
\end{equation}
If
\[
 n\ge\max(r_a+s_a,r_b+s_b),
\]
the two supports are disjoint and
\begin{equation}\label{eq:total-defect-rank}
 \rank(C_n-A_nB_n)=h_L+h_R.
\end{equation}
\end{proposition}

\begin{proof}
Equation \eqref{eq:finite-product} is \eqref{eq:missing-indices}.  With
$k=1-\ell$ in the first sum, one gets
\[
 \sum_{\ell\ge1}a_{-(i+\ell-1)}b_{j+\ell-1},
\]
which is precisely the factorization $H_a^-H_b^+$.  Both factors have rank
$h_L$: suitable terminal row and column selections produce anti-triangular
$h_L\times h_L$ minors with anti-diagonal entries $a_{-r_a}$ and $b_{s_b}$.
Thus their product has rank $h_L$.  The right corner is identical after
reversing the last rows and columns, giving rank $h_R$.  When the two corner
supports are disjoint, their ranks add.
\end{proof}

\begin{remark}[Relation to Widom's product identity]
Equation~\eqref{eq:finite-product} is the finite-band specialization of
Widom's finite-section product identity, which expresses the discrepancy
between the product of two finite Toeplitz sections and the Toeplitz section
of the product symbol by two Hankel boundary terms \cite{Widom1976}; see also
\cite{BottcherSilbermann2006,Virtanen2022}.  Thus the existence of two
boundary corrections is classical.  Proposition~\ref{prop:finite-product}
records the exact rectangular supports and, for exact finite bands, their
ranks $h_L$ and $h_R$, because precisely this rank bookkeeping is what enters
the defect-index realization of the recurrence sectors.  Related algebras
of Toeplitz matrix sequences have been studied in \cite{SerraCapizzano2001}.
\end{remark}

\subsection{Defect-order filtration and graded sectors}\label{subsec:defect-filtration}

The corner-rank calculation also gives a direct matrix realization of the
graded sectors when the two factors are one-sided.  This realization differs
from recentering: instead of changing the Laurent placement of the factors, we
interpolate between their exact finite product and the clean Toeplitz section
and sort the determinant by the number of defect directions used.

We first record two technical tools.  The first says that coefficients created
by fixed corner corrections remain in the ordinary compound recurrence space;
the second separates its exterior-degree sectors by scaling one factor.

\begin{lemma}[Same-corner boundary coefficients and the compound transfer]\label{lem:fixed-boundary-compound}
Let
\[
 g(z)\defeq\sum_{\nu=-m}^{p}g_\nu z^\nu,
 \qquad g_{-m}g_p\ne0,
 \qquad q_g(z)\defeq z^m g(z),
\]
and put
\[
 \mathcal T_g\defeq(-1)^p g_p\,\bigwedge^p C_{q_g}.
\]
Suppose that each $E^{(\alpha)}_n$ is supported, for all sufficiently large
$n$, either in a fixed upper-left corner rectangle or in a fixed lower-right
corner rectangle, with entries independent of $n$ in the corresponding local
corner coordinates.  Then every coefficient of
\[
 \det\!\left(\T_n(g)+x_1E^{(1)}_n+\cdots+x_ME^{(M)}_n\right)
\]
is, up to a fixed shift of the size parameter, a scalar matrix coefficient of
a power of $\mathcal T_g$ over the rational function field in the band
coefficients.  Consequently each such coefficient sequence is annihilated by
$\operatorname{char}(\mathcal T_g)$, and the annihilation identity persists
under coefficient specialization.
\end{lemma}

\begin{proof}
Determinant multilinearity expresses each coefficient as a fixed linear
combination of complementary Toeplitz minors.  Selecting $k$ entries from an
upper-left correction deletes $k$ rows and $k$ columns at the left boundary;
selecting $\ell$ entries from a lower-right correction deletes $\ell$ rows
and $\ell$ columns at the right boundary.  Thus every nonzero complementary
minor has zero row-column displacement separately at the two ends.  This
same-corner balance is the point that fails for arbitrary ``bounded boundary''
perturbations.

For such a balanced minor, finitely many Laplace steps consume the fixed left
boundary geometry without changing the Laurent placement of the remaining
Toeplitz bulk.  What remains is a fixed linear combination of normalized
right-boundary states of the standard row-column transfer of the companion
paper \cite{AlekseyevKhomovskyRecurrences2026}; Appendix~\ref{app:boundary-closure}
spells out this membership step and its limitation.  Hence the coefficient
sequence is a scalar matrix coefficient of the standard row-column transfer.
Over the rational function field, the companion paper identifies that transfer
with the normalized compound transfer $\mathcal T_g$ up to similarity.
Therefore the coefficient is a scalar matrix coefficient of
$\mathcal T_g^{\,n-n_0}$ for a fixed $n_0$.  Cayley--Hamilton gives the
annihilation statement.  Since that recurrence identity is polynomial in the
band and corner coefficients, it persists under specialization.
\end{proof}

For an indeterminate $u$, scale only the first factor by
\begin{equation}\label{eq:weight-dilation}
 a^{\langle u\rangle}(z)
 \defeq
 \sum_{\nu=-r_a}^{s_a}u^{s_a-\nu}a_\nu z^\nu
 =u^{s_a}a(z/u).
\end{equation}
Its polynomial core satisfies
\begin{equation}\label{eq:weight-core-dilation}
 q_a^{\langle u\rangle}(z)
 =u^{d_a}q_a(z/u),
 \qquad d_a\defeq r_a+s_a,
\end{equation}
so every root of $q_a$ is multiplied by $u$, while the leading coefficient
$a_{s_a}$ is unchanged.

For a nonzero rational function $f(u)$, write $\nu_0(f)$ for its order at
$u=0$ and
\[
 \deg_\infty f\defeq-\operatorname{ord}_{u=\infty}f.
\]
For a polynomial, $\deg_\infty$ is its ordinary degree.

\begin{lemma}[Exterior-degree weight separation]\label{lem:weight-separation}
Assume that $q_a$ and $q_b$ lie on the Zariski-open locus where they are
coprime, have simple roots, and all recurrence modes in the relevant exterior
degree are distinct.  Put $S\defeq s_a+s_b$, and under the splitting
$C_{q_aq_b}\sim C_{q_a}\oplus C_{q_b}$ write
\begin{equation}\label{eq:weight-sector-space}
 W_\delta
 \defeq
 \bigwedge^{s_a+\delta}V_a
 \otimes
 \bigwedge^{s_b-\delta}V_b.
\end{equation}
Let $X_n(u)$ be a scalar matrix coefficient of the normalized $S$th compound
transfer for $q_a^{\langle u\rangle}q_b$, decomposed into its components on
the spaces $W_\delta$.  Suppose that, for constants independent of $n$,
\begin{equation}\label{eq:weight-slope-bounds}
 \nu_0(X_n)\ge \alpha n-O(1),
 \qquad
 \deg_\infty X_n\le \beta n+O(1).
\end{equation}
Then a nonzero sector component can occur only for
\begin{equation}\label{eq:weight-window}
 \alpha\le s_a+\delta\le\beta.
\end{equation}
\end{lemma}

\begin{proof}
The dilation \eqref{eq:weight-core-dilation} multiplies every recurrence mode
in $W_\delta$ by $u^{w_\delta}$, where
$w_\delta\defeq s_a+\delta$.  Hence, defining
\[
 Y_{\delta,n}(u)\defeq u^{-w_\delta n}X_{\delta,n}(u),
\]
the nonzero functions $Y_{\delta,n}$ range in a fixed
finite-dimensional space of rational functions.  Their orders at $0$ and
$\infty$ are therefore bounded independently of $n$.

Let $w_{\max}$ be the largest weight of a nonzero sector.  For infinitely many
$n$ its contribution is nonzero, and for all sufficiently large such $n$ no
smaller weight can cancel its highest $u$-degree.  Thus
$\deg_\infty X_n=w_{\max}n+O(1)$, so
$w_{\max}\le\beta$ by \eqref{eq:weight-slope-bounds}.  The same argument at
$u=0$ for the smallest nonzero weight gives $w_{\min}\ge\alpha$, proving
\eqref{eq:weight-window}.
\end{proof}

Let
\begin{equation}\label{eq:one-sided-factors}
 a(z)\defeq\sum_{i=0}^{r}a_{-i}z^{-i},
 \qquad
 b(z)\defeq\sum_{j=0}^{s}b_jz^j,
 \qquad
 a_{-r}a_0b_0b_s\ne0,
\end{equation}
with $r,s\ge1$.  Set
\[
 q_a(z)\defeq z^ra(z),
 \qquad q_b(z)\defeq b(z),
 \qquad c(z)\defeq a(z)b(z),
 \qquad h\defeq\min(r,s),
\]
and write
\[
 A_n\defeq\T_n(a),\qquad
 B_n\defeq\T_n(b),\qquad
 C_n\defeq\T_n(c),\qquad
 E_n\defeq C_n-A_nB_n.
\]
Here $A_n$ is lower triangular, $B_n$ is upper triangular, and
Proposition~\ref{prop:finite-product} leaves only the upper-left defect, of
rank $h$ for all sufficiently large $n$.  Introduce an indeterminate $x$ and
expand
\begin{equation}\label{eq:defect-interpolation}
 \mathcal F_n(x)
 \defeq
 \det(A_nB_n+xE_n)
 =\sum_{k=0}^{h}x^kF_{k,n}.
\end{equation}
Thus
\begin{equation}\label{eq:defect-interpolation-endpoints}
 F_{0,n}=(a_0b_0)^n,
 \qquad
 \sum_{k=0}^{h}F_{k,n}=D_n(c).
\end{equation}
For $0\le j\le h$ define the sector polynomial
\begin{equation}\label{eq:defect-sector-polynomial}
 \Psi_j(t)
 \defeq
 \bigl(\chi_{q_a,j}\boxtimes\chi_{q_b,s-j}\bigr)(t),
\end{equation}
and for $0\le k\le h$ put
\begin{equation}\label{eq:defect-cumulative-polynomial}
 \Phi_k(t)\defeq\prod_{j=0}^{k}\Psi_j(t).
\end{equation}
Theorem~\ref{thm:graded-product} gives
$\Phi_h=\chi_{q_aq_b,s}$.  Let $L$ denote the forward shift,
$(Lx)_n\defeq x_{n+1}$.

\begin{proposition}[Exterior realization of the defect filtration]
\label{prop:defect-exterior-filtration}
Assume temporarily that $q_a$ and $q_b$ are coprime, have simple roots,
and have pairwise distinct recurrence modes in the relevant exterior degree;
work over a splitting field.  Let $V_a\defeq K^r$ and $V_b\defeq K^s$,
with companion operators $C_a\defeq C_{q_a}$ and $C_b\defeq C_{q_b}$.  For
$0\le j\le h$ put
\begin{equation}\label{eq:defect-sector-space}
 W_j
 \defeq
 \bigwedge^jV_a\otimes\bigwedge^{s-j}V_b,
 \qquad
 \mathcal W_{\le k}\defeq\bigoplus_{j=0}^{k}W_j.
\end{equation}
Under the Chinese-remainder identification
$C_{q_aq_b}\sim C_a\oplus C_b$, the normalized exterior-power transfer
\begin{equation}\label{eq:defect-full-transfer}
 \mathcal T
 \defeq
 (-1)^sa_0b_s\,\bigwedge^s(C_a\oplus C_b)
\end{equation}
preserves every $\mathcal W_{\le k}$.  Its restriction to $W_j$ is similar to
\begin{equation}\label{eq:defect-sector-transfer}
 \mathcal S_j
 \defeq
 a_0b_0\,
 \bigl(\bigwedge^jC_a\bigr)
 \otimes
 \bigl(\bigwedge^jC_b^{-T}\bigr),
\end{equation}
and therefore has characteristic polynomial $\Psi_j$.

For every $0\le k\le h$, in the stable range of $n$, the sequence
$(F_{k,n})_n$ is a scalar matrix coefficient of the restricted transfer
$\mathcal T|_{\mathcal W_{\le k}}$.  Equivalently, it is a sum of scalar
matrix coefficients, one from each $W_j$ with $0\le j\le k$.  Thus the
standard filtration
\begin{equation}\label{eq:defect-standard-filtration}
 0\subseteq W_0\subseteq
 W_0\oplus W_1\subseteq\cdots\subseteq
 \bigoplus_{j=0}^{h}W_j
 =\bigwedge^s(V_a\oplus V_b)
\end{equation}
is the exterior-algebra realization of the defect-order filtration.
\end{proposition}

\begin{proof}
The usual exterior-power decomposition gives
\[
 \bigwedge^s(V_a\oplus V_b)
 \simeq
 \bigoplus_{j=0}^{h}
 \left(\bigwedge^jV_a\otimes\bigwedge^{s-j}V_b\right).
\]
For the $V_b$ factor, Hodge duality identifies
\begin{equation}\label{eq:defect-hodge}
 \bigwedge^{s-j}C_b
 \sim
 (\det C_b)\,\bigwedge^jC_b^{-T}.
\end{equation}
Since $\det C_b=(-1)^sb_0/b_s$, multiplying by the normalization in
\eqref{eq:defect-full-transfer} gives precisely
\eqref{eq:defect-sector-transfer}.  Hence
$\operatorname{char}(\mathcal T|_{W_j})=\Psi_j$ by the graded multiplication
law.

It remains to locate the scalar coefficient $F_k$ in this decomposition.
Since
\[
 A_nB_n+xE_n=C_n-(1-x)E_n
\]
and $E_n$ is supported in a fixed corner window, determinant multilinearity
expresses $F_{k,n}$ as a fixed linear combination of boundary minors of
$C_n$.  Lemma~\ref{lem:fixed-boundary-compound} therefore places $F_k$ in the
full $s$th compound recurrence space
$\bigoplus_{j=0}^{h}W_j$.

We now use the dilation \eqref{eq:weight-dilation}.  In the present
one-sided situation $s_a=0$, so
\[
 a^{\langle u\rangle}(z)
 =\sum_{i=0}^{r}u^i a_{-i}z^{-i}.
\]
Let
\[
 I_n\defeq\{1,\ldots,n\},
 \qquad
 J^-\defeq\{1-h,\ldots,0\}.
\]
Weighted Cauchy--Binet for
$A_nB_n+xE_n$ gives
\begin{equation}\label{eq:one-sided-weighted-CB}
 F_{k,n}^{\langle u\rangle}
 =
 \sum_{\substack{K\subseteq J^-\cup I_n,\ |K|=n\\
                   |K\cap J^-|=k}}
 \det\mathcal T(a^{\langle u\rangle})[I_n,K]\,
 \det\mathcal T(b)[K,I_n].
\end{equation}
For every term,
\begin{equation}\label{eq:one-sided-minor-weight}
 \det\mathcal T(a^{\langle u\rangle})[I_n,K]
 =u^{\Sigma I_n-\Sigma K}
  \det\mathcal T(a)[I_n,K],
\end{equation}
where $\Sigma K$ denotes the sum of the elements of $K$.  Put
$P_-\defeq K\cap J^-$ and $H\defeq I_n\setminus K$.  Then
$|P_-|=|H|=k$ and
\[
 \Sigma I_n-\Sigma K=\Sigma H-\Sigma P_-.
\]
The set $P_-$ ranges in a fixed finite interval, while
$\Sigma H$ is between $O(1)$ and $kn+O(1)$.  Hence
\begin{equation}\label{eq:one-sided-weight-bounds}
 \nu_0(F_{k,n}^{\langle u\rangle})\ge-O(1),
 \qquad
 \deg_\infty(F_{k,n}^{\langle u\rangle})\le kn+O(1).
\end{equation}
By Lemma~\ref{lem:weight-separation}, a nonzero sector component can
therefore have only exterior-degree weight $j$ with $0\le j\le k$.  Thus
$F_{k,n}$ is a scalar matrix coefficient of
$\mathcal T|_{\mathcal W_{\le k}}$, as claimed.
\end{proof}

\begin{theorem}[Defect-order filtration]\label{thm:defect-order-filtration}
For every $0\le k\le h$, the polynomial $\Phi_k$ annihilates the sequence
$(F_{k,n})_n$.  For generic coefficients of the two factors, $\Phi_k$ is the
minimal recurrence polynomial of this sequence.  Consequently its generic
minimal recurrence order is
\begin{equation}\label{eq:defect-order-degree}
 \deg\Phi_k
 =\sum_{j=0}^{k}\binom rj\binom sj.
\end{equation}
Moreover, for $k\ge1$, the filtered sequence
\begin{equation}\label{eq:defect-associated-graded-sequence}
 \Phi_{k-1}(L)F_k
\end{equation}
has generic minimal recurrence polynomial $\Psi_k$.  Thus the successive
associated graded pieces of the defect-order filtration are exactly the graded
recurrence sectors.
\end{theorem}

\begin{proof}
On the coprime simple-root locus,
Proposition~\ref{prop:defect-exterior-filtration} and Cayley--Hamilton give
\begin{equation}\label{eq:defect-filtration-annihilation}
 \Phi_k(L)F_k=0,
\end{equation}
because
\[
\operatorname{char}(\mathcal T|_{\mathcal W_{\le k}})
 =\prod_{j=0}^{k}\Psi_j
 =\Phi_k.
\]
The recurrence identity is algebraic in the coefficients of the two factors,
so it extends from this Zariski-dense locus to every exact one-sided pair in
\eqref{eq:one-sided-factors}.

It remains to prove generic sharpness.  Introduce independent root coordinates
\begin{equation}\label{eq:defect-root-coordinates}
 a(z)\defeq a_0\prod_{\alpha=1}^{r}(1-\rho_\alpha z^{-1}),
 \qquad
 b(z)\defeq b_0\prod_{\beta=1}^{s}(1-\tau_\beta z).
\end{equation}
For $|I|=|J|=j$, the corresponding mode of $\Psi_j$ is
\begin{equation}\label{eq:defect-mode-base}
 \lambda_{I,J}
 \defeq
 a_0b_0
 \prod_{\alpha\in I}\rho_\alpha
 \prod_{\beta\in J}\tau_\beta.
\end{equation}
Indeed, the degree-$j$ $q_a$ mode is
$(-1)^ja_0\prod_{\alpha\in I}\rho_\alpha$, while indexing a
degree-$(s-j)$ $q_b$ mode by the complement of $J$ gives
$(-1)^jb_0\prod_{\beta\in J}\tau_\beta$.  In the independent root
coordinates all modes \eqref{eq:defect-mode-base} are distinct monomials.

Fix one such mode with $j\le k$ and write
\begin{equation}\label{eq:defect-fixed-mode}
 \lambda\defeq\lambda_{I,J}.
\end{equation}
Suppose that this mode were absent generically from $F_k$.  Since $\Phi_k$ is
squarefree over the rational function field in the root coordinates, the
proper divisor
\begin{equation}\label{eq:defect-mode-removed}
 P_\lambda(t)\defeq\frac{\Phi_k(t)}{t-\lambda}
\end{equation}
would also annihilate $F_k$.

For $k=0$ there is nothing to prove, so assume $k\ge1$.  Choose sets
$R\subseteq[r]$ and $S\subseteq[s]$ of cardinality $k$ with $I\subseteq R$
and $J\subseteq S$, and specialize
\begin{equation}\label{eq:defect-degree-reduction}
 \rho_\alpha=0\quad(\alpha\notin R),
 \qquad
 \tau_\beta=0\quad(\beta\notin S).
\end{equation}
The actual supports of the specialized factors shrink to $[-k,0]$ and
$[0,k]$; denote the resulting symbols by $a^*$ and $b^*$ and put
$c^*\defeq a^*b^*$.  Their defect is supported in the upper-left
$k\times k$ corner and has rank $k$.  The coefficient of $x^k$ in
\eqref{eq:defect-interpolation} therefore selects the whole defect block and
leaves the trailing clean Toeplitz section.  Since the two Hankel factors in the proof of
Proposition~\ref{prop:finite-product} are anti-triangular at full size,
\begin{equation}\label{eq:defect-top-layer-specialization}
 F^*_{k,n}
 =(a^*_{-k}b^*_k)^kD_{n-k}(c^*),
 \qquad n\ge k.
\end{equation}
The scalar in front is nonzero.

The remaining $2k$ roots are generic.  Proposition
\ref{prop:background-recurrence} and Theorem~\ref{thm:graded-product} therefore
show that the minimal recurrence polynomial of $D_{n-k}(c^*)$, and hence of
$F^*_{k,n}$, has the nonzero roots
\begin{equation}\label{eq:defect-specialized-modes}
 a_0b_0
 \prod_{\alpha\in I'}\rho_\alpha
 \prod_{\beta\in J'}\tau_\beta,
 \qquad
 I'\subseteq R,\quad J'\subseteq S,
 \quad |I'|=|J'|\le k.
\end{equation}
Call this minimal polynomial $\Phi_k^*(t)$.  Under the specialization
\eqref{eq:defect-degree-reduction}, every mode of $\Phi_k$ that uses a root
outside $R$ or $S$ becomes zero, while the nonzero modes are exactly
\eqref{eq:defect-specialized-modes}.  Hence for some $M\ge0$,
\begin{equation}\label{eq:defect-specialized-removed}
 P_\lambda^*(t)
 =t^M\frac{\Phi_k^*(t)}{t-\lambda}.
\end{equation}
Since $\Phi_k^*(0)\ne0$, the polynomial $\Phi_k^*$ does not divide
$P_\lambda^*$.  Thus $P_\lambda^*$ cannot annihilate $F_k^*$, contradicting
the specialization of the assumed identity $P_\lambda(L)F_k=0$.  Every mode
of $\Phi_k$ therefore occurs generically in $F_k$, proving that $\Phi_k$ is
the generic minimal recurrence polynomial.  Taking degrees gives
\eqref{eq:defect-order-degree}.

Finally, the filtration in Proposition~\ref{prop:defect-exterior-filtration}
has associated graded piece
\begin{equation}\label{eq:defect-associated-graded-space}
 \mathcal W_{\le k}/\mathcal W_{\le k-1}
 \simeq W_k,
\end{equation}
whose characteristic polynomial is $\Psi_k$.  Applying
$\Phi_{k-1}(L)$ to $F_k$ kills all lower-sector matrix coefficients.  Generic
sharpness shows that every mode in $W_k$ occurs and that none is killed by
$\Phi_{k-1}$.  Hence the filtered sequence has generic minimal recurrence
polynomial $\Psi_k$.
\end{proof}

\begin{remark}[Compound-map form of the graded pieces]
\label{rem:defect-compound-map}
Hodge duality also gives the intrinsic identification
\[
 W_j
 \simeq
 \operatorname{Hom}\!\left(\bigwedge^jV_b,\bigwedge^jV_a\right)
 \otimes\det(V_b).
\]
Thus the $j$th associated graded piece may be viewed as the space of $j$th
compound maps across the two one-sided factors.  In this language, defect
order at most $k$ retains compound orders $0,\ldots,k$, while passing to the
associated graded quotient isolates compound order $k$.
\end{remark}

\begin{remark}[Balanced defect orders]\label{rem:balanced-defect-orders}
When $r=s=h$, Theorem~\ref{thm:defect-order-filtration} gives
\[
 \operatorname{ord}_{\rm gen}(F_k)
 =\sum_{j=0}^{k}\binom hj^2.
\]
For $h=2$ these orders are $1,5,6$, while for $h=3$ they are
$1,10,19,20$.  At $k=h$, Vandermonde's identity gives the full generic order
$\binom{2h}{h}$.  Thus defect order defines a filtration of the full recurrence
space whose associated graded dimensions are the sector dimensions
$\binom hk^2$.
\end{remark}

\section{One-diagonal growth and the Pascal rule}\label{sec:one-diagonal}

The general corner-defect formula simplifies sharply when one factor is
bidiagonal.  This case is useful when the lower or upper semibandwidth is
increased one diagonal at a time.  Let
\[
 a(z)\defeq\sum_{\nu=-r}^{s}a_\nu z^\nu,
 \qquad
 q(z)\defeq z^r a(z),
 \qquad
 d\defeq r+s,
\]
with exact support $[-r,s]$.  We extend the coefficient sequence by
$a_\nu\defeq0$ outside this support.  Throughout the two-sector discussion
below we assume $r,s\ge1$; at an outer edge the nonexistent adjacent sector is
simply absent.  Take $\alpha,\beta\in K$ with $\alpha\beta\ne0$ and put
\begin{equation}\label{eq:bidiagonal-symbols}
 b_+(z)\defeq\alpha+\beta z,
 \qquad
 c_+(z)\defeq a(z)b_+(z),
 \qquad
 b_-(z)\defeq\alpha+\beta z^{-1},
 \qquad
 c_-(z)\defeq a(z)b_-(z).
\end{equation}
Thus $c_+$ has semibandwidths $(r,s+1)$, whereas $c_-$ has
semibandwidths $(r+1,s)$.  Let $e_1,\ldots,e_n$ denote the standard coordinate
vectors of $K^n$.

\begin{proposition}[Rank-one defect for one-diagonal growth]\label{prop:one-diagonal-defect}
For $n\ge1$, set
\[
 u_n\defeq(a_{-1},a_{-2},\ldots,a_{-n})^T,
 \qquad
 v_n\defeq(a_n,a_{n-1},\ldots,a_1)^T.
\]
Then
\begin{align}
 \T_n(a)\T_n(b_+)
 &=\T_n(c_+)-\beta u_ne_1^T,
 \label{eq:upper-rank-one-defect}\\
 \T_n(a)\T_n(b_-)
 &=\T_n(c_-)-\beta v_ne_n^T.
 \label{eq:lower-rank-one-defect}
\end{align}
Hence each finite-section defect has rank at most one; it has rank one when
the corresponding vector is nonzero.
\end{proposition}

\begin{proof}
For the upper factor, the $(i,j)$ entry of the product is
\[
 \alpha a_{j-i}+\beta a_{j-i-1}
\]
when $j>1$, which is exactly the $(i,j)$ entry of $\T_n(c_+)$.  In the first
column the second term is missing, leaving the correction
$\beta(a_{-1},\ldots,a_{-n})^T e_1^T$.  The lower formula is obtained in the
same way from the last column.
\end{proof}

For a monic polynomial
\[
 f(t)\defeq\prod_{\nu=1}^m(t-\lambda_\nu)
\]
and a scalar $\gamma\in K$, write
\begin{equation}\label{eq:root-scaling}
 (\mathsf S_\gamma f)(t)
 \defeq
 \prod_{\nu=1}^m(t-\gamma\lambda_\nu).
\end{equation}
Thus $\mathsf S_\gamma$ scales every recurrence mode by $\gamma$.

\begin{theorem}[One-diagonal Pascal rule]\label{thm:one-diagonal-pascal}
The canonical recurrence polynomials for the two enlarged bands satisfy
\begin{align}
 \chi_{q(\alpha+\beta z),\,s+1}
 &=
 (\mathsf S_\alpha\chi_{q,s})
 (\mathsf S_\beta\chi_{q,s+1}),
 \label{eq:upper-pascal}\\
 \chi_{q(\beta+\alpha z),\,s}
 &=
 (\mathsf S_\alpha\chi_{q,s})
 (\mathsf S_\beta\chi_{q,s-1}).
 \label{eq:lower-pascal}
\end{align}
Consequently, their degrees split according to Pascal's identities
\begin{align}
 \binom{d+1}{s+1}
 &=\binom ds+\binom d{s+1},
 \label{eq:upper-pascal-degree}\\
 \binom{d+1}{s}
 &=\binom ds+\binom d{s-1}.
 \label{eq:lower-pascal-degree}
\end{align}
\end{theorem}

\begin{proof}
The polynomial core of $c_+$ is $q(z)(\alpha+\beta z)$.  For the linear core
$\ell_+(z)\defeq\alpha+\beta z$, its two profile entries are
\[
 \chi_{\ell_+,0}(t)=t-\beta,
 \qquad
 \chi_{\ell_+,1}(t)=t-\alpha.
\]
Theorem~\ref{thm:graded-product} at total exterior degree $s+1$ therefore has
exactly two sectors, giving \eqref{eq:upper-pascal}.  For $c_-$ the new
polynomial core is $q(z)(\beta+\alpha z)$; applying the same argument at
exterior degree $s$ gives \eqref{eq:lower-pascal}.  Taking degrees yields
\eqref{eq:upper-pascal-degree}--\eqref{eq:lower-pascal-degree}.
\end{proof}

The factor $\mathsf S_\alpha\chi_{q,s}$ in both formulas is inherited from the
old determinant sequence.  The other factor is the new adjacent sector.  The
matrix product selects these sectors asymmetrically.  Indeed,
\begin{equation}\label{eq:bidiagonal-determinant-product}
 \det\!\bigl(\T_n(a)\T_n(b_\pm)\bigr)
 =\alpha^n D_n(a),
\end{equation}
because both bidiagonal factors have determinant $\alpha^n$.  In particular,
with the natural unit-diagonal normalization $\alpha\defeq1$, the determinant
of the distorted enlarged-band matrix is exactly the old determinant
sequence $D_n(a)$.  Thus the apparent mismatch of recurrence orders is not a
delay before a larger recurrence begins: the special boundary data suppress
the adjacent sector, so the larger canonical annihilator is nonminimal for
this distorted sequence.

A direct cofactor expansion also shows why this rank-one simplification does
not, in general, close on the clean determinant sequence alone.  Writing
$C_n^+\defeq\T_n(c_+)$, one has
\begin{equation}\label{eq:one-column-cofactor-expansion}
 \det\!\bigl(\T_n(a)\T_n(b_+)\bigr)
 =D_n(c_+)
 -\beta\sum_{j=1}^{\min(r,n)}
 a_{-j}(-1)^{j+1}
 \det C_n^+[\widehat j\mid\widehat 1],
\end{equation}
where $C_n^+[\widehat j\mid\widehat 1]$ denotes the matrix obtained by
deleting row $j$ and column $1$.  The term $j=1$ is the principal Toeplitz
determinant $D_{n-1}(c_+)$, but for $j>1$ the remaining matrix is a
nonprincipal Toeplitz minor.  Thus only the extreme case $r=1$ closes
immediately on clean Toeplitz determinants of different sizes; in general,
additional boundary-minor states are unavoidable.

The missing sector can nevertheless be located directly at matrix level.  Let
\[
 \mathcal T(a)\defeq(a_{j-i})_{i,j\in\mathbb Z},
 \qquad
 I_n\defeq\{1,\ldots,n\},
\]
and, for $0\le j\le n$, put
\[
 K_{n,j}\defeq\{0,1,\ldots,n\}\setminus\{j\},
 \qquad
 M_{n,j}(a)\defeq
 \det\mathcal T(a)[I_n,K_{n,j}].
\]

\begin{proposition}[Cauchy--Binet chain]\label{prop:cauchy-binet-chain}
For upper one-diagonal growth,
\begin{equation}\label{eq:cauchy-binet-chain}
 D_n(c_+)
 =
 \sum_{j=0}^n
 \alpha^{\,n-j}\beta^j M_{n,j}(a).
\end{equation}
The two endpoint minors are
\begin{equation}\label{eq:cauchy-binet-endpoints}
 M_{n,0}(a)=D_n(a),
 \qquad
 M_{n,n}(a)=D_n(za).
\end{equation}
The second endpoint is the standard determinant sequence for the adjacent
profile component $(q,s+1)$.
\end{proposition}

\begin{proof}
On bi-infinite matrices one has
$\mathcal T(c_+)=\mathcal T(a)\mathcal T(b_+)$.  In the Cauchy--Binet expansion
of its minor with row and column set $I_n$, the bidiagonal factor allows only
$n$-element intermediate sets contained in $\{0,\ldots,n\}$.  They are
exactly the sets $K_{n,j}$.  The corresponding bidiagonal minor has determinant
$\alpha^{n-j}\beta^j$, giving \eqref{eq:cauchy-binet-chain}.  The first endpoint
uses the original columns $1,\ldots,n$, while the last uses $0,\ldots,n-1$,
which is precisely the Toeplitz section of $za$.
\end{proof}

\begin{remark}[Boundary-minor recurrence closure]\label{rem:boundary-minor-closure}
Two different boundary mechanisms occur here.  The cofactors in
\eqref{eq:one-column-cofactor-expansion} delete a row and a column at the
same boundary.  They therefore have zero row-column displacement and belong,
after normalization, to the standard row-column state module of the companion
paper \cite{AlekseyevKhomovskyRecurrences2026}.  Appendix~\ref{app:boundary-closure}
proves this standard-state membership for the same-corner minors generated by
finite-section multiplication and then uses generic observability to eliminate
them in favor of consecutive clean determinants.

The Cauchy--Binet chain in Proposition~\ref{prop:cauchy-binet-chain} also
shows why one must not make the corresponding assertion for an arbitrary
bounded boundary geometry.  Its endpoint $M_{n,n}(a)=D_n(za)$ has undergone a
Laurent recentering and belongs to the adjacent profile rather than to the
standard state module for $a$.  No generic clean-shift reduction to $D_n(a)$
is asserted for such cross-boundary displacement.  The defect filtrations in
the main text use only the same-corner coefficient sequences covered by
Lemma~\ref{lem:fixed-boundary-compound}.
\end{remark}

The whole rank-one correction contains intermediate boundary minors as well as
the adjacent endpoint, so it need not itself satisfy only the adjacent-sector
recurrence.  The adjacent sector emerges after the inherited one is filtered
out.  We write $\operatorname{adj}(M)$ for the classical adjugate of a square
matrix $M$.

\begin{proposition}[Filtering the rank-one defect]\label{prop:defect-filter}
Let $L$ denote the forward shift on sequences,
\[
 (Lx)_n\defeq x_{n+1}.
\]
For upper one-diagonal growth set
\[
 X_n\defeq\alpha^nD_n(a),
 \qquad
 Y_n\defeq D_n(c_+),
 \qquad
 H_n\defeq e_1^T\operatorname{adj}(\T_n(c_+))u_n,
\]
and
\[
 P(t)\defeq(\mathsf S_\alpha\chi_{q,s})(t),
 \qquad
 Q(t)\defeq(\mathsf S_\beta\chi_{q,s+1})(t).
\]
Then
\begin{equation}\label{eq:defect-decomposition}
 Y_n=X_n+\beta H_n
\end{equation}
and
\begin{equation}\label{eq:defect-filter}
 Q(L)P(L)H=0.
\end{equation}
Equivalently, $P(L)H$ is annihilated by the adjacent-sector polynomial $Q$.
For generic coefficients and parameters, $Q$ is its minimal recurrence
polynomial.
\end{proposition}

\begin{proof}
Equation~\eqref{eq:defect-decomposition} is the rank-one determinant identity
applied to \eqref{eq:upper-rank-one-defect}.  The sequence $X$ is annihilated
by $P$, whereas Theorem~\ref{thm:one-diagonal-pascal} shows that $Y$ is
annihilated by $PQ$.  Applying $P(L)$ to
$Y-X=\beta H$ and then applying $Q(L)$ gives
\eqref{eq:defect-filter}.  Generically the two sector mode sets are disjoint
and all adjacent-sector amplitudes survive, giving the final statement.
\end{proof}

\begin{corollary}[One-sided extreme cases]\label{cor:one-sided-growth}
The upper-growth identity simplifies, for $r=1$, to
\begin{equation}\label{eq:r1-growth}
 D_n(c_+)-\beta a_{-1}D_{n-1}(c_+)
 =\alpha^nD_n(a),
 \qquad n\ge1,
\end{equation}
and the adjacent-sector factor is $t-\beta a_{-1}$.  Symmetrically, for
$s=1$,
\begin{equation}\label{eq:s1-growth}
 D_n(c_-)-\beta a_sD_{n-1}(c_-)
 =\alpha^nD_n(a),
 \qquad n\ge1,
\end{equation}
and the adjacent-sector factor is $t-\beta a_s$.
\end{corollary}

\begin{proof}
For $r=1$, one has $u_n=a_{-1}e_1$, so the rank-one cofactor in
\eqref{eq:defect-decomposition} is $a_{-1}D_{n-1}(c_+)$.  Moreover,
$s+1=d$, hence $\chi_{q,s+1}(t)=t-a_{-1}$ by
\eqref{eq:extreme-profiles}.  The lower case follows by symmetry from
$v_n=a_se_n$ and $\chi_{q,0}(t)=t-a_s$.
\end{proof}

The size of the adjacent sector in upper growth is
$\binom d{s+1}=\binom d{r-1}$, whereas in lower growth it is
$\binom d{s-1}$.  Thus repeated upper growth is especially simple while the
lower semibandwidth remains $1$, and repeated lower growth is especially
simple while the upper semibandwidth remains $1$.

\begin{remark}[Zero diagonal]
The choice $\alpha=0$ is qualitatively different.  For example,
$b_+(z)=\beta z$ is singular as a finite section, and, when $r\ge1$, the
support changes from $[-r,s]$ to $[-r+1,s+1]$: one lower diagonal is traded for
one upper diagonal rather than increasing the total bandwidth.  Its polynomial
core is only rescaled, since
\[
 z^{r-1}(\beta z a(z))=\beta q(z).
\]
Thus a zero diagonal produces an extreme recentering, not genuine
one-diagonal growth.
\end{remark}

\section{Recentring and defect indices}\label{sec:defect-indices}

The same product $c=ab$ can be represented by factors shifted in
opposite directions relative to the main diagonal.  For an integer $\delta$,
put
\begin{equation}\label{eq:recenter}
 a^{(\delta)}(z)\defeq z^\delta a(z),
 \qquad
 b^{(\delta)}(z)\defeq z^{-\delta}b(z).
\end{equation}
Then $a^{(\delta)}b^{(\delta)}=c$.  We call this opposite shift of the two
factors a \emph{recentring}, because it changes which coefficients of the
factors lie on the main diagonal without changing their product.  Writing each ordered pair as (lower semibandwidth, upper semibandwidth), the
shifted factors have
\[
 a^{(\delta)}:\ (r_a-\delta,s_a+\delta),
 \qquad
 b^{(\delta)}:\ (r_b+\delta,s_b-\delta).
\]
Hence both are genuine finite Laurent bands exactly when
\begin{equation}\label{eq:delta-range}
 -h_R\le\delta\le h_L.
\end{equation}

\begin{proposition}[Boundary-rank ladder]\label{prop:rank-ladder}
For every admissible $\delta$, the two corner defects in
\[
 \T_n(a^{(\delta)})\T_n(b^{(\delta)})-\T_n(c)
\]
have, for sufficiently large $n$, ranks
\begin{equation}\label{eq:rank-ladder}
 h_L(\delta)\defeq h_L-\delta,
 \qquad
 h_R(\delta)\defeq h_R+\delta.
\end{equation}
In particular,
\[
 h_L(\delta)+h_R(\delta)=h_L+h_R,
\]
so increasing $\delta$ by one transfers one unit of defect rank from the
left corner to the right corner.
\end{proposition}

\begin{proof}
Apply \eqref{eq:hLhR} to the shifted semibandwidths:
\[
 \min(r_a-\delta,s_b-\delta)=h_L-\delta,
\]
\[
 \min(s_a+\delta,r_b+\delta)=h_R+\delta.
\]
\end{proof}

Proposition~\ref{prop:rank-ladder} motivates the term \emph{defect index}
for $\delta$: changing $\delta$ by one transfers one unit of defect rank from
one boundary to the other while keeping the total defect rank fixed.

The same integer $\delta$ also indexes the exterior sectors.  Set
\[
 q_a(z)\defeq z^{r_a}a(z),
 \qquad
 q_b(z)\defeq z^{r_b}b(z),
\]
with degrees $d_a\defeq r_a+s_a$ and $d_b\defeq r_b+s_b$.  The polynomial cores do not
change under recentering:
\[
 z^{r_a-\delta}a^{(\delta)}(z)=q_a(z),
 \qquad
 z^{r_b+\delta}b^{(\delta)}(z)=q_b(z).
\]
Only their exterior degrees change:
\begin{equation}\label{eq:shifted-degrees}
 k_a\defeq s_a+\delta,
 \qquad
 k_b\defeq s_b-\delta.
\end{equation}
Their sum is the upper semibandwidth
\[
 S\defeq s_a+s_b
\]
of the product symbol $c$.

\begin{theorem}[Defect-index factorization]\label{thm:defect-index}
For every admissible integer $\delta$ in \eqref{eq:delta-range}, set
\begin{equation}\label{eq:psi-delta}
 \Psi_\delta(t)
 \defeq
 \chi_{q_a,s_a+\delta}\boxtimes
 \chi_{q_b,s_b-\delta}.
\end{equation}
Then the sequence
\begin{equation}\label{eq:sector-sequence}
 Z_n^{(\delta)}
 \defeq
 \det\T_n(z^\delta a)\,
 \det\T_n(z^{-\delta}b)
\end{equation}
is annihilated by $\Psi_\delta$, and the full recurrence polynomial of the
unperturbed Toeplitz determinant sequence $\det\T_n(c)$ satisfies
\begin{equation}\label{eq:defect-product}
 \chi_{q_aq_b,S}(t)
 =
 \prod_{\delta=-h_R}^{h_L}\Psi_\delta(t).
\end{equation}
Generically, each $\Psi_\delta$ is the minimal recurrence polynomial of
$Z_n^{(\delta)}$.
\end{theorem}

\begin{proof}
By \eqref{eq:standard-symbol} and \eqref{eq:shifted-degrees}, the two factors
in \eqref{eq:sector-sequence} are precisely the standard Toeplitz determinant
problems for $(q_a,s_a+\delta)$ and $(q_b,s_b-\delta)$.  Their generic Widom mode sets are therefore encoded by the recurrence polynomials
$\chi_{q_a,s_a+\delta}$ and $\chi_{q_b,s_b-\delta}$.  Termwise products of
the modes have pairwise-product bases, so $Z_n^{(\delta)}$ is annihilated by
\eqref{eq:psi-delta}.  Finally, Theorem~\ref{thm:graded-product} with
$k=S$ has sectors $i=s_a+\delta$, and the admissible range of $i$ is exactly
\eqref{eq:delta-range}.  This gives \eqref{eq:defect-product}.
\end{proof}

\begin{remark}
The theorem changes the interpretation of the ``missing'' sectors.  They do
not arise most naturally as individual terms in a cofactor expansion of one
fixed product $\T_n(a)\T_n(b)$.  Instead, each sector is realized by a
different finite-section product in the recentered family
\[
 c=(z^\delta a)(z^{-\delta}b).
\]
All members have the same Toeplitz bulk but different allocations of the
fixed total boundary-defect rank.
\end{remark}

\subsection{Two-corner interval filtration}\label{subsec:two-corner-filtration}

The one-sided filtration above starts from an exact finite product with a
single corner defect.  For two-sided factors, both defects are present, and
the two defect orders move through the same sector ladder from opposite
directions.  This gives a bivariate refinement that connects the fixed-product
picture directly to the recentering sectors of Theorem~\ref{thm:defect-index}.

Retain the notation of Proposition~\ref{prop:finite-product} and put
\[
 d_a\defeq r_a+s_a,
 \qquad
 d_b\defeq r_b+s_b,
 \qquad
 S\defeq s_a+s_b.
\]
For sufficiently large $n$, introduce independent variables $x,y$ and expand
\begin{equation}\label{eq:bivariate-defect-interpolation}
 \mathcal F_n(x,y)
 \defeq
 \det(A_nB_n+xL_n+yR_n)
 =
 \sum_{k=0}^{h_L}\sum_{\ell=0}^{h_R}
 x^k y^\ell F_{k,\ell,n}.
\end{equation}
The degree bounds follow from the corner ranks.  At the two endpoints,
\begin{equation}\label{eq:bivariate-defect-endpoints}
 F_{0,0,n}=\det A_n\det B_n,
 \qquad
 \sum_{k=0}^{h_L}\sum_{\ell=0}^{h_R}F_{k,\ell,n}=D_n(c).
\end{equation}
Thus $k$ records left-corner defect order and $\ell$ records right-corner
defect order.

For an admissible defect index $-h_R\le\delta\le h_L$, retain the sector
polynomial $\Psi_\delta$ from \eqref{eq:psi-delta}.  For
$0\le k\le h_L$ and $0\le\ell\le h_R$, set
\begin{equation}\label{eq:bivariate-cumulative-polynomial}
 \Phi_{k,\ell}(t)
 \defeq
 \prod_{\delta=-\ell}^{k}\Psi_\delta(t).
\end{equation}

\begin{proposition}[Two-boundary exterior filtration]
\label{prop:two-boundary-exterior-filtration}
Assume temporarily that $q_a$ and $q_b$ are coprime, have simple roots,
and have pairwise distinct recurrence modes in exterior degree $S$; work over
a splitting field.  Let $V_a\defeq K^{d_a}$ and
$V_b\defeq K^{d_b}$ carry the companion operators $C_{q_a}$ and $C_{q_b}$.
For every admissible $\delta$, put
\begin{equation}\label{eq:bivariate-sector-space}
 W_\delta
 \defeq
 \bigwedge^{s_a+\delta}V_a
 \otimes
 \bigwedge^{s_b-\delta}V_b,
\end{equation}
and define the interval subspace
\begin{equation}\label{eq:bivariate-interval-space}
 \mathcal W_{k,\ell}
 \defeq
 \bigoplus_{\delta=-\ell}^{k}W_\delta.
\end{equation}
Under the Chinese-remainder identification
$C_{q_aq_b}\sim C_{q_a}\oplus C_{q_b}$, the normalized transfer
\begin{equation}\label{eq:bivariate-full-transfer}
 \mathcal T
 \defeq
 (-1)^S a_{s_a}b_{s_b}\,
 \bigwedge^S(C_{q_a}\oplus C_{q_b})
\end{equation}
preserves every $\mathcal W_{k,\ell}$, and
\begin{equation}\label{eq:bivariate-sector-charpoly}
 \operatorname{char}(\mathcal T|_{W_\delta})=\Psi_\delta.
\end{equation}
In the stable range of $n$, the sequence $(F_{k,\ell,n})_n$ is a scalar
matrix coefficient of $\mathcal T|_{\mathcal W_{k,\ell}}$.
Consequently $\Phi_{k,\ell}$ annihilates this sequence.
\end{proposition}

\begin{proof}
The exterior-power decomposition is
\[
 \bigwedge^S(V_a\oplus V_b)
 \simeq
 \bigoplus_{\delta=-h_R}^{h_L}
 \left(
  \bigwedge^{s_a+\delta}V_a
  \otimes
  \bigwedge^{s_b-\delta}V_b
 \right),
\]
and the normalization in \eqref{eq:bivariate-full-transfer} is the one from
Proposition~\ref{prop:background-recurrence}.  The graded multiplication law
therefore gives \eqref{eq:bivariate-sector-charpoly}.

We first place $F_{k,\ell}$ in the full compound space.  From
\eqref{eq:finite-product},
\[
 A_nB_n+xL_n+yR_n
 =C_n-(1-x)L_n-(1-y)R_n.
\]
The two corrections are supported in fixed boundary windows, so determinant
multilinearity expresses every coefficient $F_{k,\ell,n}$ as a fixed linear
combination of boundary minors of $C_n$.  By
Lemma~\ref{lem:fixed-boundary-compound}, $F_{k,\ell}$ is therefore a scalar
matrix coefficient of the full transfer \eqref{eq:bivariate-full-transfer}.
It remains only to determine which sector blocks can occur.

Let
\[
 I_n\defeq\{1,\ldots,n\},
 \qquad
 J_n^-\defeq\{1-h_L,\ldots,0\},
 \qquad
 J_n^+\defeq\{n+1,\ldots,n+h_R\}.
\]
The missing-index formula may be weighted before Cauchy--Binet: give an
intermediate index weight $x$ on $J_n^-$, weight $1$ on $I_n$, and weight
$y$ on $J_n^+$.  Hence, after the dilation
\eqref{eq:weight-dilation},
\begin{equation}\label{eq:bivariate-weighted-CB}
 F_{k,\ell,n}^{\langle u\rangle}
 =
 \sum_{\substack{K\subseteq J_n^-\cup I_n\cup J_n^+,\ |K|=n\\
                   |K\cap J_n^-|=k,\ |K\cap J_n^+|=\ell}}
 \det\mathcal T(a^{\langle u\rangle})[I_n,K]\,
 \det\mathcal T(b)[K,I_n].
\end{equation}
For every term, row and column factorization gives
\begin{equation}\label{eq:bivariate-minor-weight}
 \det\mathcal T(a^{\langle u\rangle})[I_n,K]
 =u^{ns_a+\Sigma I_n-\Sigma K}
  \det\mathcal T(a)[I_n,K].
\end{equation}
Put
\[
 P_-\defeq K\cap J_n^-,
 \qquad
 P_+\defeq K\cap J_n^+,
 \qquad
 H\defeq I_n\setminus K.
\]
Then $|P_-|=k$, $|P_+|=\ell$, and $|H|=k+\ell$.  Put
$Q_+\defeq P_+-n\subseteq\{1,\ldots,h_R\}$.  A direct sum calculation
gives
\begin{equation}\label{eq:bivariate-weight-exponent}
 \Sigma I_n-\Sigma K
 =-\ell n+\Sigma H-\Sigma P_- -\Sigma Q_+.
\end{equation}
The last two sums are bounded independently of $n$, while
\[
 \Sigma H=O(1)
 \quad\text{at its lower extreme},
 \qquad
 \Sigma H=(k+\ell)n+O(1)
 \quad\text{at its upper extreme}.
\]
Therefore
\begin{equation}\label{eq:bivariate-weight-bounds}
 \nu_0(F_{k,\ell,n}^{\langle u\rangle})
 \ge (s_a-\ell)n-O(1),
 \qquad
 \deg_\infty(F_{k,\ell,n}^{\langle u\rangle})
 \le (s_a+k)n+O(1).
\end{equation}

On $W_\delta$, the exterior-degree weight is $s_a+\delta$.  Applying
Lemma~\ref{lem:weight-separation} to
\eqref{eq:bivariate-weight-bounds} shows that a nonzero component must satisfy
\[
 s_a-\ell\le s_a+\delta\le s_a+k,
\]
or equivalently $-\ell\le\delta\le k$.  Thus
$F_{k,\ell,n}$ is a scalar matrix coefficient of
$\mathcal T|_{\mathcal W_{k,\ell}}$.  Cayley--Hamilton then gives the
annihilator \eqref{eq:bivariate-cumulative-polynomial}.
\end{proof}

\begin{lemma}[Support-reducing root degeneration]
\label{lem:bivariate-root-degeneration}
Use the root coordinates
\[
 q_a(z)=a_{s_a}\prod_{\alpha=1}^{d_a}(z-\rho_\alpha),
 \qquad
 q_b(z)=b_{s_b}\prod_{\beta=1}^{d_b}(z-\sigma_\beta).
\]
Fix $0\le k\le h_L$, $0\le\ell\le h_R$, and a sector index
$-\ell\le\delta\le k$.  Let
\[
 |I|=s_a+\delta,
 \qquad
 |J|=s_b-\delta,
\]
and choose
\begin{align*}
 Z_a&\subseteq I^c, & |Z_a|&=r_a-k,
 &M_a&\subseteq I, & |M_a|&=s_a-\ell,\\
 Z_b&\subseteq J^c, & |Z_b|&=r_b-\ell,
 &M_b&\subseteq J, & |M_b|&=s_b-k.
\end{align*}
Define deformed cores by
\begin{align}
 q_{a,\varepsilon}(z)
 &\defeq
 \varepsilon^{s_a-\ell}a_{s_a}
 \prod_{\alpha\in Z_a}(z-\varepsilon\rho_\alpha)
 \prod_{\alpha\in M_a}(z-\varepsilon^{-1}\rho_\alpha)
 \prod_{\alpha\notin Z_a\cup M_a}(z-\rho_\alpha),
 \label{eq:qa-eps}\\
 q_{b,\varepsilon}(z)
 &\defeq
 \varepsilon^{s_b-k}b_{s_b}
 \prod_{\beta\in Z_b}(z-\varepsilon\sigma_\beta)
 \prod_{\beta\in M_b}(z-\varepsilon^{-1}\sigma_\beta)
 \prod_{\beta\notin Z_b\cup M_b}(z-\sigma_\beta).
 \label{eq:qb-eps}
\end{align}
Then the Laurent symbols
$a_\varepsilon(z)\defeq z^{-r_a}q_{a,\varepsilon}(z)$ and
$b_\varepsilon(z)\defeq z^{-r_b}q_{b,\varepsilon}(z)$ have coefficientwise
limits $a^*,b^*$ with exact supports
\begin{equation}\label{eq:bivariate-reduced-supports}
 a^*:[-k,\ell],
 \qquad
 b^*:[-\ell,k].
\end{equation}
Moreover, a degree-$(s_a+\delta')$ mode of $q_{a,\varepsilon}$ indexed by
$I'$ has $\varepsilon$-valuation
\begin{equation}\label{eq:a-mode-valuation}
 (s_a-\ell)-|I'\cap M_a|+|I'\cap Z_a|,
\end{equation}
which is zero exactly when
$M_a\subseteq I'$ and $I'\cap Z_a=\varnothing$.  The analogous statement for
$q_{b,\varepsilon}$ has valuation
\begin{equation}\label{eq:b-mode-valuation}
 (s_b-k)-|J'\cap M_b|+|J'\cap Z_b|.
\end{equation}
Consequently the nonzero surviving product modes, as
$-\ell\le\delta'\le k$, are exactly the modes of the reduced pair
$a^*,b^*$, with exterior degrees $\ell+\delta'$ and $k-\delta'$, respectively.
\end{lemma}

\begin{proof}
In \eqref{eq:qa-eps}, pair one power of the prefactor $\varepsilon$ with
each factor indexed by $M_a$.  Taking $\varepsilon\to0$ gives
\[
 q_a^*(z)
 =a_{s_a}(-1)^{s_a-\ell}
  \left(\prod_{\alpha\in M_a}\rho_\alpha\right)
  z^{r_a-k}
  \prod_{\alpha\notin Z_a\cup M_a}(z-\rho_\alpha).
\]
Its lowest nonzero degree is $r_a-k$ and its highest is $r_a+\ell$;
after multiplication by $z^{-r_a}$ this is exactly the first support in
\eqref{eq:bivariate-reduced-supports}.  The calculation for $q_b^*$ is
identical and gives lowest and highest degrees $r_b-\ell$ and $r_b+k$.

For a subset mode, the leading-coefficient prefactor contributes the first
term in \eqref{eq:a-mode-valuation}; a selected root from $M_a$ contributes
$-1$ to the valuation and a selected root from $Z_a$ contributes $+1$.
Since $|M_a|=s_a-\ell$, the valuation is nonnegative and vanishes exactly
under the stated containment and disjointness conditions.  Removing the
mandatory set $M_a$ from a surviving $(s_a+\delta')$-subset leaves
$\ell+\delta'$ roots of the minimal core of $a^*$; similarly one obtains
$k-\delta'$ roots for $b^*$.  The factor
$(-1)^{s_a-\ell}\prod_{M_a}\rho_\alpha$ in the leading coefficient of
$q_a^*$, and its analogue for $b^*$, show by direct substitution in
\eqref{eq:chi-qk} that the surviving mode values agree exactly, including
signs and normalization.
\end{proof}

\begin{lemma}[Top bidefect coefficient]\label{lem:bivariate-top-layer}
Let $a^*$ and $b^*$ have exact supports $[-k,\ell]$ and $[-\ell,k]$,
respectively, and let $F^*_{k,\ell,n}$ denote the coefficient of $x^ky^\ell$
in their interpolation \eqref{eq:bivariate-defect-interpolation}.  Then, for
$n\ge k+\ell$,
\begin{equation}\label{eq:bivariate-top-layer}
 F^*_{k,\ell,n}
 =
 (a^*_{-k}b^*_k)^k
 (a^*_{\ell}b^*_{-\ell})^\ell
 D_{n-k-\ell}(a^*b^*).
\end{equation}
\end{lemma}

\begin{proof}
For these reduced supports the left defect is a square $k\times k$ block and
the right defect is a square $\ell\times\ell$ block.  In the coefficient of
$x^ky^\ell$, determinant multilinearity must therefore select both complete
corner blocks.  The remaining rows and columns are the common middle
principal interval $\{k+1,\ldots,n-\ell\}$.  Both defects vanish there, so
the remaining block of $A_n^*B_n^*$ equals the corresponding block of
$T_n(a^*b^*)$, namely $T_{n-k-\ell}(a^*b^*)$.

For the left defect, the two Hankel factors in the proof of
Proposition~\ref{prop:finite-product} are $k\times k$ anti-triangular matrices
with anti-diagonal entries $a^*_{-k}$ and $b^*_k$.  Their anti-triangular
signs cancel, giving determinant $(a^*_{-k}b^*_k)^k$.  The reversed right
corner gives in the same way $(a^*_{\ell}b^*_{-\ell})^\ell$.  Since the two
selected row sets equal the two selected column sets, the complementary
Laplace sign is positive.  Multiplying the three determinants proves
\eqref{eq:bivariate-top-layer}.
\end{proof}

\begin{theorem}[Bivariate defect-order filtration]
\label{thm:bivariate-defect-filtration}
For every $0\le k\le h_L$ and $0\le\ell\le h_R$,
\begin{equation}\label{eq:bivariate-annihilation}
 \Phi_{k,\ell}(L)F_{k,\ell}=0.
\end{equation}
For generic coefficients of $a$ and $b$, $\Phi_{k,\ell}$ is the minimal
recurrence polynomial of $(F_{k,\ell,n})_n$.  Hence the generic minimal
recurrence order is
\begin{equation}\label{eq:bivariate-order}
 \operatorname{ord}_{\rm gen}(F_{k,\ell})
 =
 \sum_{\delta=-\ell}^{k}
 \binom{d_a}{s_a+\delta}
 \binom{d_b}{s_b-\delta}.
\end{equation}
Moreover, for $k\ge1$ the sequence
\begin{equation}\label{eq:bivariate-horizontal-filter}
 \Phi_{k-1,\ell}(L)F_{k,\ell}
\end{equation}
has generic minimal recurrence polynomial $\Psi_k$, while for $\ell\ge1$
\begin{equation}\label{eq:bivariate-vertical-filter}
 \Phi_{k,\ell-1}(L)F_{k,\ell}
\end{equation}
has generic minimal recurrence polynomial $\Psi_{-\ell}$.
\end{theorem}

\begin{proof}
The annihilation statement follows from
Proposition~\ref{prop:two-boundary-exterior-filtration}; as before, the
identity extends from the Zariski-dense distinct-mode locus because it is
algebraic in the coefficients.

For generic sharpness, use the root coordinates from
Lemma~\ref{lem:bivariate-root-degeneration}.  Fix a sector
$-\ell\le\delta\le k$ and a mode indexed by
\[
 I\subseteq[d_a],\qquad |I|=s_a+\delta,
 \qquad
 J\subseteq[d_b],\qquad |J|=s_b-\delta.
\]
Its value is
\begin{equation}\label{eq:bivariate-fixed-mode}
 \lambda
 \defeq
 (-1)^S a_{s_a}b_{s_b}
 \prod_{\alpha\in I}\rho_\alpha
 \prod_{\beta\in J}\sigma_\beta.
\end{equation}
Choose the sets $Z_a,M_a,Z_b,M_b$ from
Lemma~\ref{lem:bivariate-root-degeneration}; the required cardinalities are
possible exactly because $-\ell\le\delta\le k$.  Under the resulting
$\varepsilon$-degeneration, the chosen mode has a finite nonzero limit
$\lambda^*$, while the complete set of nonzero surviving modes is exactly the
mode set of the reduced pair
\[
 a^*:[-k,\ell],
 \qquad
 b^*:[-\ell,k].
\]
Thus the nonzero part of the specialized polynomial $\Phi_{k,\ell}$ is
\begin{equation}\label{eq:bivariate-reduced-polynomial}
 \Phi^*(t)
 \defeq
 \chi_{q_{a^*}q_{b^*},\,k+\ell}(t),
\end{equation}
and all other roots specialize to zero.

By Lemma~\ref{lem:bivariate-top-layer},
\[
 F^*_{k,\ell,n}
 =
 (a^*_{-k}b^*_k)^k
 (a^*_{\ell}b^*_{-\ell})^\ell
 D_{n-k-\ell}(a^*b^*).
\]
The scalar prefactor is nonzero.  Since $a^*b^*$ has exact balanced support
$[-(k+\ell),k+\ell]$, Proposition~\ref{prop:background-recurrence} shows that
$\Phi^*$ is generically the minimal recurrence polynomial of this specialized
sequence.

Suppose that the original mode $\lambda$ were absent generically from
$F_{k,\ell}$.  On the generic simple-mode locus, the proper divisor
\[
 P_\lambda(t)\defeq\frac{\Phi_{k,\ell}(t)}{t-\lambda}
\]
would also annihilate $F_{k,\ell}$.  Lemma~\ref{lem:bivariate-root-degeneration}
shows that the specialization has no poles and, for some $M\ge0$, gives
\begin{equation}\label{eq:bivariate-removed-specialization}
 P_\lambda^*(t)
 =t^M\frac{\Phi^*(t)}{t-\lambda^*}.
\end{equation}
But $\Phi^*(0)\ne0$, so $\Phi^*$ does not divide $P_\lambda^*$.  This
contradicts the minimality of $\Phi^*$ for the top-bidefect sequence.  Every
mode in every sector $-\ell\le\delta\le k$ therefore occurs generically,
proving both generic minimality and \eqref{eq:bivariate-order}.

Finally, Proposition~\ref{prop:two-boundary-exterior-filtration} gives
\[
 \mathcal W_{k,\ell}/\mathcal W_{k-1,\ell}\simeq W_k,
 \qquad
 \mathcal W_{k,\ell}/\mathcal W_{k,\ell-1}\simeq W_{-\ell}.
\]
Applying the preceding cumulative sector factors kills exactly the earlier
blocks, while generic sharpness guarantees that every mode of the newly added
block survives.  This proves the two filtered statements.
\end{proof}

\begin{remark}[Interval rather than a new bidegree]
\label{rem:bivariate-interval}
The bivariate filtration does not create an additional mixed sector.  Instead,
left and right defect orders extend one interval of the existing recentering
ladder from opposite ends:
\[
 \mathcal W_{k,\ell}
 =W_{-\ell}\oplus\cdots\oplus W_k.
\]
For $k,\ell\ge1$,
\begin{equation}\label{eq:no-mixed-bidegree}
 \mathcal W_{k,\ell}
 =
 \mathcal W_{k-1,\ell}+\mathcal W_{k,\ell-1},
\end{equation}
so the corresponding double associated-graded quotient is zero.  The two boundary directions therefore recover the same one-dimensional defect-index
ladder already produced by Laurent recentering.  At the top bidegree,
\[
 \Phi_{h_L,h_R}=\chi_{q_aq_b,S},
\]
so the full interval recovers the entire canonical recurrence polynomial.
When one corner rank is zero, the interval has only one moving endpoint and
the construction reduces, after reindexing, to the one-sided defect-order
filtration of Theorem~\ref{thm:defect-order-filtration}.
\end{remark}

\section{The quadratic-core pentadiagonal prototype}\label{sec:pentadiagonal}

Five-diagonal Toeplitz determinants have long been treated by explicit
formulas, generating functions, and recurrence methods
\cite{MarrVineyard1988,Sweet,JiaYangLi2016}, and factorizations of
five-diagonal matrices through tridiagonal Toeplitz factors have also been
used in numerical linear algebra \cite{DieleLopez1998}.  The purpose of the
present prototype is therefore not to rederive the five-diagonal theory, but
to show how its six generic recurrence modes split canonically into the three
graded sectors $1+4+1$ and how the two extreme sectors arise as boundary
defect indices of recentered factorizations.

Let
\begin{equation}\label{eq:tri-symbols}
 a(z)\defeq uz^{-1}+a_0+vz,
 \qquad
 b(z)\defeq xz^{-1}+b_0+yz.
\end{equation}
Their polynomial cores are the quadratic polynomials
\[
 q_a(z)\defeq u+a_0z+vz^2,
 \qquad
 q_b(z)\defeq x+b_0z+yz^2.
\]
Their product
\[
 c(z)\defeq a(z)b(z)=\sum_{j=-2}^2c_jz^j
\]
has polynomial core
\[
 q_c(z)\defeq z^2c(z)=q_a(z)q_b(z)
\]
and coefficients
\begin{align}
 c_{-2}&=ux,
 &c_{-1}&=ub_0+a_0x,\notag\\
 c_0&=uy+a_0b_0+vx,
 &c_1&=a_0y+vb_0,
 &c_2&=vy.
 \label{eq:penta-coeffs}
\end{align}
Put
\[
 A_n\defeq\T_n(a),\qquad B_n\defeq\T_n(b),\qquad C_n\defeq\T_n(c),
\]
and
\[
 \kappa_L\defeq uy,
 \qquad
 \kappa_R\defeq vx.
\]
Let $E_{ij}$ denote the matrix unit with a single $1$ in position $(i,j)$
and zeros elsewhere.  Proposition~\ref{prop:finite-product} becomes the exact
identity
\begin{equation}\label{eq:tri-product-corners}
 A_nB_n=C_n-\kappa_L E_{11}-\kappa_R E_{nn}.
\end{equation}

Let
\[
 X_n\defeq\det A_n,
 \qquad
 Y_n\defeq\det B_n,
 \qquad
 D_n\defeq\det C_n,
\]
with $D_0\defeq1$ and $D_{-1}\defeq0$.  The location of the two defects makes
this case unusually simple.  In the multilinear expansion of
$\det(C_n-\kappa_LE_{11}-\kappa_RE_{nn})$, one may select neither corner
correction, exactly one of them, or both.  Selecting the correction at
$(1,1)$ deletes the first row and column of the clean Toeplitz section and
leaves $C_{n-1}$; selecting the correction at $(n,n)$ does the same at the
opposite end; selecting both leaves $C_{n-2}$.  Thus all four contributions
remain within the same clean Toeplitz determinant sequence:
\begin{align}
 X_nY_n
 &=D_n-\kappa_LD_{n-1}-\kappa_RD_{n-1}
   +\kappa_L\kappa_RD_{n-2}\notag\\
 &=D_n-(\kappa_L+\kappa_R)D_{n-1}
   +\kappa_L\kappa_R D_{n-2}.
 \label{eq:scalar-filter}
\end{align}
This closure on principal Toeplitz sections is special to defects placed at
diagonal corners.  A deeper one-column correction produces nonprincipal
same-corner Toeplitz minors, as shown by
\eqref{eq:one-column-cofactor-expansion}; Appendix~\ref{app:boundary-closure}
shows that these standard-sector minors can generically be eliminated in
favor of a bounded block of consecutive clean determinant shifts.  By
contrast, the endpoint $M_{n,n}(a)=D_n(za)$ of the Cauchy--Binet chain in
Proposition~\ref{prop:cauchy-binet-chain} is a recentered adjacent-profile
sequence and is deliberately outside that statement.  In the present
pentadiagonal identity the diagonal-corner geometry is already exceptional
enough to give the explicit $D_{n-1}$ and $D_{n-2}$ closure in
\eqref{eq:scalar-filter}.
Equivalently, let $L$ denote the forward-shift operator, defined by
$L D_n\defeq D_{n+1}$.  Then
\begin{equation}\label{eq:shift-filter}
 X_{n+2}Y_{n+2}=(L-\kappa_L)(L-\kappa_R)D_n.
\end{equation}
For two power series, define their Hadamard product coefficientwise by
\[
 \left(\sum_{n\ge0}x_nz^n\right)\odot
 \left(\sum_{n\ge0}y_nz^n\right)
 \defeq
 \sum_{n\ge0}x_ny_nz^n.
\]
At the generating-function level, the scalar identity becomes
\begin{equation}\label{eq:hadamard-gf}
 \sum_{n\ge0}D_nz^n
 =
 \frac{
   \left(\sum_{n\ge0}X_nz^n\right)
   \odot
   \left(\sum_{n\ge0}Y_nz^n\right)
 }{(1-\kappa_L z)(1-\kappa_R z)}.
\end{equation}

The standard order-two determinant recurrences for $A_n$ and $B_n$ have recurrence polynomials
\begin{equation}\label{eq:fg}
 f(t)\defeq t^2-a_0t+uv,
 \qquad
 g(t)\defeq t^2-b_0t+xy.
\end{equation}
Their composed product is
\begin{align}
 (f\boxtimes g)(t)
 ={}&t^4-a_0b_0t^3
 +(a_0^2xy+b_0^2uv-2uvxy)t^2\notag\\
 &-a_0b_0uvxy\,t+(uvxy)^2.
 \label{eq:quartic-sector}
\end{align}
The three sectors $\delta=-1,0,1$ in
Theorem~\ref{thm:defect-index} have dimensions $1,4,1$.  Using
\eqref{eq:extreme-profiles}, they give
\begin{equation}\label{eq:1-4-1}
 \chi_{q_c,2}(t)
 =(t-\kappa_L)\,(f\boxtimes g)(t)\,(t-\kappa_R).
\end{equation}
Thus the familiar generic degree six is decomposed as
\[
 6=1+4+1.
\]

Expanding \eqref{eq:1-4-1} in the product coefficients
\eqref{eq:penta-coeffs} gives
\begin{align}
 \chi_{q_c,2}(t)
 ={}&t^6-c_0t^5
 +(c_{-1}c_1-c_{-2}c_2)t^4\notag\\
 &-(c_{-2}c_1^2+c_{-1}^2c_2-2c_{-2}c_0c_2)t^3\notag\\
 &+c_{-2}c_2(c_{-1}c_1-c_{-2}c_2)t^2\notag\\
 &-c_0(c_{-2}c_2)^2t+(c_{-2}c_2)^3.
 \label{eq:sweet-poly}
\end{align}
This is the recurrence polynomial of Sweet's order-six recurrence
\cite{Sweet}.  Hence the pentadiagonal recurrence can be reconstructed from
two quadratic recurrences, their composed product, and the two one-dimensional
boundary sectors.

\begin{remark}[Exterior interpretation]
Let $V_a\defeq K^2$ and $V_b\defeq K^2$ be the vector spaces on which the
companion matrices of $q_a$ and $q_b$ are viewed as linear maps.  Then
\[
 \bigwedge^2(V_a\oplus V_b)
 \simeq
 \bigwedge^2V_a
 \oplus(V_a\otimes V_b)
 \oplus\bigwedge^2V_b.
\]
The three summands have dimensions $1,4,1$ and correspond exactly to the
three recenterings $\delta=-1,0,1$.
\end{remark}

\section{Arbitrary products and the defect-index polytope}\label{sec:multifactor}

Let
\begin{equation}\label{eq:many-factors}
 a_j(z)\defeq\sum_{\nu=-r_j}^{s_j}a_{j,\nu}z^\nu,
 \qquad j=1,\ldots,m,
\end{equation}
where $[-r_j,s_j]$ is the exact support interval of $a_j$.  Put
\[
 d_j\defeq r_j+s_j,
 \qquad
 q_j(z)\defeq z^{r_j}a_j(z),
\]
and
\[
 c(z)\defeq\prod_{j=1}^m a_j(z),
 \qquad
 Q(z)\defeq\prod_{j=1}^m q_j(z).
\]
Let
\[
 R\defeq\sum_{j=1}^m r_j,
 \qquad
 S\defeq\sum_{j=1}^m s_j,
 \qquad
 D\defeq R+S=\sum_{j=1}^m d_j.
\]
Then $Q(z)=z^Rc(z)$, and the Toeplitz determinant sequence
$\det\T_n(c)$ corresponds to exterior degree $S$.

\subsection{Multi-index sectors}

Choose integers $\delta_1,\ldots,\delta_m$ and recenter
\[
 \widetilde a_j(z)\defeq z^{\delta_j}a_j(z).
\]
The product of the recentered symbols remains $c$ exactly when
\begin{equation}\label{eq:delta-sum-zero}
 \sum_{j=1}^m\delta_j=0.
\end{equation}
The $j$th shifted factor remains a finite Laurent band exactly when
\begin{equation}\label{eq:delta-local-bounds}
 -s_j\le\delta_j\le r_j.
\end{equation}
Set
\begin{equation}\label{eq:k-from-delta}
 k_j\defeq s_j+\delta_j.
\end{equation}
Then \eqref{eq:delta-sum-zero}--\eqref{eq:delta-local-bounds} become
\begin{equation}\label{eq:multi-index-sector}
 0\le k_j\le d_j,
 \qquad
 \sum_{j=1}^m k_j=S.
\end{equation}
Thus admissible recenterings are in bijection with the finite set
\begin{equation}\label{eq:KS}
 \cK_S
 \defeq
 \left\{
 (k_1,\ldots,k_m)\in\mathbb Z^m:
 0\le k_j\le d_j,
 \ \sum_jk_j=S
 \right\}.
\end{equation}
Moreover,
\begin{equation}\label{eq:standard-factor-multi}
 z^{\delta_j}a_j(z)=z^{k_j-d_j}q_j(z),
\end{equation}
so each recentered factor is exactly the standard Laurent representative of
the degree-$k_j$ profile of $q_j$.

\begin{theorem}[Multi-factor recurrence decomposition]\label{thm:multi-factor}
Take $\mathbf k\in\cK_S$, with components $k_1,\ldots,k_m$, and set
\begin{equation}\label{eq:Psi-k}
 \Psi_{\mathbf k}(t)
 \defeq
 \chi_{q_1,k_1}\boxtimes\cdots\boxtimes\chi_{q_m,k_m}.
\end{equation}
Then
\begin{equation}\label{eq:multi-factor-product}
 \chi_{Q,S}(t)
 =\prod_{\mathbf k\in\cK_S}\Psi_{\mathbf k}(t).
\end{equation}
The determinant product associated with the recentering $\mathbf k$,
\begin{equation}\label{eq:multi-sector-sequence}
 Z_n^{(\mathbf k)}
 \defeq\prod_{j=1}^m
 \det\T_n\!\left(z^{k_j-d_j}q_j\right),
\end{equation}
is annihilated by $\Psi_{\mathbf k}$ and generically has this minimal
recurrence polynomial.
\end{theorem}

\begin{proof}
Iterate Theorem~\ref{thm:graded-product}.  Equivalently, partition an
$S$-subset of the roots of $Q$ according to how many selected roots come
from each labelled factor $q_j$.  The resulting occupation numbers are
precisely the points of $\cK_S$, and the modes in a fixed sector are the
products encoded by \eqref{eq:Psi-k}.  Equation \eqref{eq:multi-sector-sequence}
follows from the same termwise-product argument used in
Theorem~\ref{thm:defect-index}.
\end{proof}

The dimension of a sector is
\begin{equation}\label{eq:sector-degree}
 \deg\Psi_{\mathbf k}
 =\prod_{j=1}^m\binom{d_j}{k_j}.
\end{equation}
Summing over all sectors gives the multivariate Vandermonde identity
\begin{equation}\label{eq:multi-vandermonde}
 \sum_{\mathbf k\in\cK_S}
 \prod_{j=1}^m\binom{d_j}{k_j}
 =\binom{D}{S}.
\end{equation}
By contrast, the number of sectors is obtained by coefficient extraction.
Write $[x^S]F(x)$ for the coefficient of $x^S$ in $F(x)$.  Then
\begin{equation}\label{eq:number-sectors}
 |\cK_S|
 =[x^S]\prod_{j=1}^m(1+x+\cdots+x^{d_j}).
\end{equation}

\subsection{Interface defect indices}

The shifts $\delta_j$ satisfy one linear relation, so it is natural to use
$m-1$ cumulative coordinates
\begin{equation}\label{eq:eta-def}
 \eta_j\defeq\delta_1+\cdots+\delta_j,
 \qquad j=1,\ldots,m-1,
\end{equation}
with
\[
 \eta_0\defeq0,\qquad \eta_m\defeq0.
\]
Write $\boldsymbol\eta\defeq(\eta_1,\ldots,\eta_{m-1})$.  Then
\begin{equation}\label{eq:delta-eta}
 \delta_j=\eta_j-\eta_{j-1},
\end{equation}
and
\begin{equation}\label{eq:partial-product-shift}
 \prod_{\ell=1}^j\widetilde a_\ell(z)
 =z^{\eta_j}\prod_{\ell=1}^j a_\ell(z).
\end{equation}
Thus $\eta_j$ is the net Laurent shift transported across the cut between
factors $j$ and $j+1$.  We call $\eta_j$ the \emph{interface defect index} at that
cut.

The local band constraints become
\begin{equation}\label{eq:defect-constraints}
 -s_j\le\eta_j-\eta_{j-1}\le r_j,
 \qquad j=1,\ldots,m,
\end{equation}
Hence the admissible defect indices are the integer points of the bounded
\emph{defect-index polytope}, namely the region cut out by the linear inequalities
\begin{equation}\label{eq:defect-polytope}
 \cP
 \defeq\left\{
 (\eta_1,\ldots,\eta_{m-1})\in\mathbb R^{m-1}:
 -s_j\le\eta_j-\eta_{j-1}\le r_j,\ 1\le j\le m
 \right\}.
\end{equation}
The affine map
\begin{equation}\label{eq:k-eta}
 k_j=s_j+\eta_j-\eta_{j-1},
 \qquad j=1,\ldots,m,
\end{equation}
induces a bijection
\begin{equation}\label{eq:polytope-sector-bijection}
 \cP\cap\mathbb Z^{m-1}\longleftrightarrow\cK_S.
\end{equation}

\begin{proposition}[Corner-rank meaning of the defect indices]\label{prop:cut-ranks}
Fix a cut after factor $j$, where $1\le j<m$.  Put
\[
 R_L\defeq\sum_{\ell\le j}r_\ell,
 \quad
 S_L\defeq\sum_{\ell\le j}s_\ell,
 \quad
 R_R\defeq\sum_{\ell>j}r_\ell,
 \quad
 S_R\defeq\sum_{\ell>j}s_\ell,
\]
and
\begin{equation}\label{eq:cut-h}
 h_{L,j}\defeq\min(R_L,S_R),
 \qquad
 h_{R,j}\defeq\min(S_L,R_R).
\end{equation}
For the recentered factorization determined by $\boldsymbol\eta$, the two
corner defects at this cut have ranks
\begin{equation}\label{eq:cut-rank-shift}
 h_{L,j}(\eta_j)\defeq h_{L,j}-\eta_j,
 \qquad
 h_{R,j}(\eta_j)\defeq h_{R,j}+\eta_j.
\end{equation}
Consequently
\begin{equation}\label{eq:eta-marginal}
 -h_{R,j}\le\eta_j\le h_{L,j}.
\end{equation}
\end{proposition}

\begin{proof}
Across the cut, the product of the first $j$ recentered factors is shifted by
$z^{\eta_j}$ and the product of the remaining factors by $z^{-\eta_j}$, by
\eqref{eq:partial-product-shift}.  Apply Proposition~\ref{prop:rank-ladder}
to these two partial products.
\end{proof}

The marginal intervals \eqref{eq:eta-marginal} alone do not describe the
polytope when $m>2$; the neighboring defect indices are coupled by the local
constraints \eqref{eq:defect-constraints}.  Thus, for more than two factors, the admissible defect indices are coupled by
neighboring inequalities rather than forming independent intervals.

\begin{example}[Three tridiagonal Toeplitz factors]\label{ex:three-tridiagonal}
Let $r_j=s_j=1$ for $j=1,2,3$.  Then $d_j=2$, $S=3$, and the unperturbed product
has generic minimal recurrence order
\[
 \binom63=20.
\]
There are seven multi-index sectors:
\begin{center}
\begin{tabular}{@{}ccc@{}}
\toprule
$\mathbf k$ & $(\eta_1,\eta_2)$ & $\deg\Psi_{\mathbf k}$\\
\midrule
$(0,1,2)$ & $(-1,-1)$ & $2$\\
$(0,2,1)$ & $(-1,0)$ & $2$\\
$(1,0,2)$ & $(0,-1)$ & $2$\\
$(1,1,1)$ & $(0,0)$ & $8$\\
$(1,2,0)$ & $(0,1)$ & $2$\\
$(2,0,1)$ & $(1,0)$ & $2$\\
$(2,1,0)$ & $(1,1)$ & $2$\\
\bottomrule
\end{tabular}
\end{center}
The sector dimensions sum to
\[
 8+6\cdot2=20.
\]
Thus the central zero-index sector accounts for only eight of the twenty
generic modes; the six nonzero-defect-index sectors supply the remaining twelve.
\end{example}

\section{A graded recurrence-profile algebra}\label{sec:algebra}

For a polynomial $q$ of degree $d$ with nonzero constant and leading
coefficients, set its graded recurrence profile
\begin{equation}\label{eq:profile}
 \cR(q)
 \defeq
 (\chi_{q,0},\chi_{q,1},\ldots,\chi_{q,d}).
\end{equation}
For two such profiles $F\defeq(F_0,\ldots,F_m)$ and $G\defeq(G_0,\ldots,G_n)$, set
\begin{equation}\label{eq:star}
 (F\star G)_k
 \defeq
 \prod_{i+j=k}(F_i\boxtimes G_j).
\end{equation}
For polynomial cores $q_1$ and $q_2$, Theorem~\ref{thm:graded-product} is
simply
\begin{equation}\label{eq:profile-hom}
 \cR(q_1q_2)=\cR(q_1)\star\cR(q_2).
\end{equation}
Since $\boxtimes$ is associative and commutative on root multisets, so is
$\star$.  The constant polynomial $1$ has profile $(t-1)$ and acts as the
identity.  Hence $q\mapsto\cR(q)$ is a homomorphism from multiplication of
polynomial cores to this graded recurrence-profile monoid (that is, an
associative multiplication with an identity element).

The construction interpolates between two familiar descriptions.  At one
extreme, if $Q$ is treated as a single irreducible block, the unperturbed Toeplitz
determinant recurrence is one compound object of degree $\binom DS$.  If $Q$
is split into several factors, Theorem~\ref{thm:multi-factor} resolves that
compound object into sectors indexed by defect indices.  If one passes all the way to a complete
linear factorization of $Q$, every $d_j=1$ and every admissible $k_j$ is
$0$ or $1$.  The sectors are then individual $S$-subsets of the roots, so the
factorized description reduces exactly to the individual Widom modes.

This hierarchy also gives a simple root-free recursive construction whenever a
factorization of the polynomial core is supplied.  If
\[
 Q=q_1q_2\cdots q_m,
\]
one may compute the smaller profiles $\cR(q_j)$ and combine them successively
using \eqref{eq:star}, without extracting roots or constructing the full
compound matrix at once.  Appendix~\ref{app:factorized-construction} develops
this observation for a supplied bounded-degree factorization and proves a
soft-linear arithmetic bound in the degree of the expanded output recurrence.

\section{Concluding remarks}\label{sec:conclusion}

Multiplication of banded Toeplitz data has two compatible finite-dimensional
descriptions.  Polynomial cores multiply by a graded convolution of
exterior-degree recurrence profiles, while finite Toeplitz sections multiply
up to two explicit corner defects.  Complementary exterior degrees are paired
by a scaled reciprocal duality, and opposite Laurent recenterings realize the
individual graded sectors by redistributing defect rank between the two
boundaries.

For one-sided products, defect order gives the standard filtration of
$\bigwedge^s(V_a\oplus V_b)$ by $V_a$-exterior degree, with generically sharp
cumulative recurrence polynomials.  For general two-sided factors, the two
corner corrections extend the same ladder from opposite ends: bidegree
$(k,\ell)$ contains exactly $W_{-\ell}\oplus\cdots\oplus W_k$ generically.
This interval description is generically sharp and is intrinsic rather than an
artifact of a particular cofactor expansion.  For several factors, the cumulative defect indices form a lattice
polytope whose integer points are the multi-index exterior sectors.

These results give an exact description of how finite-section boundary effects
enter determinant recurrences under structured Toeplitz multiplication.  They
also yield a root-free recursive construction of the canonical recurrence when
a useful factorization of the polynomial core is supplied.  Appendix~\ref{app:boundary-closure}
uses the companion paper's row-column state module and generic minimality to
eliminate the standard-sector, same-corner minors generated by finite-section
multiplication, while distinguishing them from cross-corner profile shifts and
from the exceptional one- and two-term corner closures.
Appendix~\ref{app:factorized-construction} proves an output-sensitive
field-operation bound for bounded-degree factorized inputs.  Natural extensions
include more general controlled finite corner perturbations, beyond the special
double-banded and moving-diagonal families for which periodicity and recurrence
phenomena have previously been studied
\cite{DuDaFonseca2022,DuDaFonsecaCorrection2024,DuFonsecaMiskolc2023,Shitov2021,Shitov2024}.
Those works concern imposed corner-supported structures, whereas the defects
considered here arise from finite-section multiplication and carry the sector
filtration described above.  Other directions include coefficient-height and
bit-complexity bounds, numerical questions, and further structural consequences
of the multi-factor defect polytope.

\section*{Declarations}
\noindent\textbf{Use of generative AI.}
During the development and preparation of this work, the authors used OpenAI's
ChatGPT for exploratory mathematical reasoning, including checks of algebraic
identities and proof strategies, as well as literature searches and assistance
with organization, drafting, and editing.  The authors independently reviewed
and verified the mathematical arguments, computations, citations, and final
text.  The authors assume responsibility for all content.

\printbibliography

\clearpage
\appendix
\clearpage
\refstepcounter{section}
\section*{Appendix \thesection: Standard-state elimination of multiplication-induced corner defects}
\label{app:boundary-closure}

The finite-section product formulas in the main text produce several kinds of
nonprincipal Toeplitz minors.  The distinction between them matters.  The
corner corrections $L_n$ and $R_n$ in Proposition~\ref{prop:finite-product}
are supported in the upper-left and lower-right corners, respectively.
Whenever determinant multilinearity selects entries from one of these
corrections, it deletes the same number of rows and columns at that same
boundary.  The cofactors in \eqref{eq:one-column-cofactor-expansion} have the
same property.  We call this \emph{same-corner balance}.  These are the
multiplication-induced minors that remain in the standard Laurent profile and
can be reduced to the standard row-column state module of the companion paper.

A bounded boundary geometry need not have this property.  A row deletion at
one end can be balanced by a column deletion at the other end; the resulting
minor then carries a nonzero profile displacement and can belong to a
recentered Toeplitz sequence instead.  Thus there are two logically distinct
steps in the elimination used here:
\begin{enumerate}
 \item same-corner balance and fixed bandwidth place the multiplication minor
       in the standard row-column state module;
 \item generic observability of that standard module expresses the resulting
       state through consecutive principal determinants.
\end{enumerate}
The second step is exactly the observability argument already used in the
companion recurrence paper.  The first is the multiplication-specific point
that we make explicit below.

Fix an exact banded symbol in the notation of the main text,
\[
 a(z)\defeq\sum_{\nu=-r_a}^{s_a}a_\nu z^\nu,
 \qquad a_{-r_a}a_{s_a}\ne0,
 \qquad D_n(a)\defeq\det\T_n(a),
\]
and put $I_n\defeq\{1,\ldots,n\}$.  For a fixed finite set
$S\subset\mathbb Z_{>0}$, write
\[
 n+1-S\defeq\{n+1-s:s\in S\}.
\]

\begin{definition}[Balanced same-corner boundary minor]\label{def:balanced-corner-minor}
A sequence $F_n$ is a balanced same-corner boundary-minor family for $a$ if,
for all sufficiently large $n$, it has the form
\begin{equation}\label{eq:balanced-corner-minor}
 F_n
 \defeq
 \det\mathcal T(a)
 \Bigl[
  I_n\setminus\bigl(R_L\cup(n+1-R_R)\bigr),\,
  I_n\setminus\bigl(C_L\cup(n+1-C_R)\bigr)
 \Bigr],
\end{equation}
where $R_L,C_L,R_R,C_R$ are fixed finite sets and
\begin{equation}\label{eq:balanced-corner-counts}
 |R_L|=|C_L|,
 \qquad
 |R_R|=|C_R|.
\end{equation}
We only use families for which the displayed minor is not identically zero.
The complementary minors created by selecting fixed numbers of entries from
the upper-left and lower-right corrections in
\eqref{eq:finite-product} satisfy \eqref{eq:balanced-corner-counts}
automatically.  The same is true for each cofactor in
\eqref{eq:one-column-cofactor-expansion}.
\end{definition}

\begin{lemma}[Standard-state membership]\label{lem:standard-state-membership}
Let $F_n$ be a balanced same-corner boundary-minor family for $a$.  Then there
exist an integer $\eta$ and a fixed row vector $\gamma^T$, independent of
$n$, such that
\begin{equation}\label{eq:defect-to-rowcolumn-state}
 F_n=\gamma^T\mathbf v_N^{\rm rc},
 \qquad N\defeq n+\eta,
\end{equation}
where $\mathbf v_N^{\rm rc}$ is the standard determinant row-column state
vector for $a$ in the companion recurrence paper
\cite{AlekseyevKhomovskyRecurrences2026}.
\end{lemma}

\begin{proof}
Choose a fixed left boundary window containing $R_L\cup C_L$.  Repeatedly
Laplace-expand an exposed row or column inside that window.  Because the
Toeplitz band has fixed width, every pivot sees only a fixed neighboring
window, so the resulting expansion tree is finite and independent of $n$.
All coefficients in that finite expansion depend only on the fixed boundary
pattern and the band coefficients.

The equality $|R_L|=|C_L|$ is what prevents a Laurent recentering during this
elimination.  After the left irregularity has been consumed, the retained row
and column indices enter the Toeplitz bulk with the same offset.  Hence every
nonzero terminal term is, after a fixed translation and a fixed shift of the
size parameter, a clean leading Toeplitz section with only a normalized
right-boundary signature.  Fixed bandwidth bounds the row and column deficits
of that signature by the lower and upper semibandwidths.  The exact signature
classification in the companion paper identifies precisely these normalized
pairs with the standard row-column states.  The fixed right-boundary pattern
$R_R,C_R$ only determines which finite set of those coordinates occurs.
Collecting the finitely many terminal terms gives
\eqref{eq:defect-to-rowcolumn-state}.
\end{proof}

\begin{remark}[Why bounded boundary geometry is not enough]
\label{rem:cross-corner-counterexample}
The same-corner balance in Definition~\ref{def:balanced-corner-minor} cannot
be dropped.  Consider the generic tridiagonal Toeplitz matrix
\[
 T_n\defeq\operatorname{tridiag}(\ell,d,u),
 \qquad D_n\defeq\det T_n.
\]
Then
\[
 D_{n+2}-dD_{n+1}+\ell uD_n=0.
\]
The cross-corner minor
\[
 F_n\defeq\det T_n[2:n,1:n-1]=\ell^{\,n-1}
\]
deletes one row at the left boundary and one column at the right boundary.
Accordingly,
\[
 F_{n+2}-dF_{n+1}+\ell uF_n
 =\ell^n(\ell-d+u),
\]
which is generically nonzero.  In the Laurent notation of this paper the
reason is transparent:
\[
 F_n=D_{n-1}(za).
\]
Thus the cross-corner displacement has changed the Laurent profile instead of
producing a state in the standard module for $a$.  The endpoint
$M_{n,n}(a)=D_n(za)$ in \eqref{eq:cauchy-binet-endpoints} is the same
recentring mechanism in the one-diagonal calculation.
\end{remark}

\begin{theorem}[Generic elimination of standard-sector corner minors]
\label{thm:generic-corner-minor-elimination}
Let $F_n$ be a balanced same-corner boundary-minor family for $a$, and put
\begin{equation}\label{eq:rowcolumn-dimension}
 N_a\defeq\binom{r_a+s_a}{s_a}.
\end{equation}
For generic coefficients of $a$, there are an integer $\eta$ and scalars
$\lambda_0,\ldots,\lambda_{N_a-1}$, independent of $n$, such that
\begin{equation}\label{eq:generic-corner-minor-elimination}
 F_n
 =\sum_{j=0}^{N_a-1}\lambda_j D_{n+\eta+j}(a)
\end{equation}
for all sufficiently large $n$.

More precisely, let $Q_a$ denote the determinant row-column transfer for the
symbol $a$ in the standard normalized state basis of the companion recurrence
paper; thus $Q_a$ is the transfer denoted $Q_{\rm rc}$ there, written in the
present paper's symbol and semibandwidth notation.  Let $e_0$ select its
principal state, and let $\mathbf v_N^{\rm rc}$ be the corresponding state
vector.  With $\eta$ and $\gamma^T$ from
Lemma~\ref{lem:standard-state-membership}, define
\begin{equation}\label{eq:observability-matrix-a}
 \mathcal O_a\defeq
 \begin{bmatrix}
  e_0^T\\
  e_0^TQ_a\\
  \vdots\\
  e_0^TQ_a^{N_a-1}
 \end{bmatrix}.
\end{equation}
The matrix $\mathcal O_a$ exists at every coefficient specialization and
records the principal determinant coordinate after
$0,1,\ldots,N_a-1$ transfer steps.  Full observability means that these
observations determine the entire standard row-column state, equivalently
that $\operatorname{rank}\mathcal O_a=N_a$.  Full observability, not existence
of the matrix, is the generic assertion.

On the generic locus, $\mathcal O_a$ is nonsingular and
\begin{equation}\label{eq:defect-observability-elimination}
 F_n
 =\gamma^T\mathcal O_a^{-1}
 \begin{bmatrix}
  D_N(a)\\
  D_{N+1}(a)\\
  \vdots\\
  D_{N+N_a-1}(a)
 \end{bmatrix},
 \qquad N\defeq n+\eta.
\end{equation}
\end{theorem}

\begin{proof}
Lemma~\ref{lem:standard-state-membership} supplies the only
multiplication-specific step: $F_n=\gamma^T\mathbf v_N^{\rm rc}$ in the
standard state module.  The remainder is the observability argument from the
older boundary-closure formulation and from the companion recurrence paper.
The companion paper classifies the standard row-column states exactly, so
$Q_a$ has dimension $N_a$, and proves generic minimal recurrence order $N_a$
for the principal determinant sequence.  Moreover,
\[
 D_{N+j}(a)=e_0^TQ_a^j\mathbf v_N^{\rm rc}
 \qquad(j\ge0),
\]
so
\begin{equation}\label{eq:clean-shift-observability}
 \begin{bmatrix}
  D_N(a)\\
  D_{N+1}(a)\\
  \vdots\\
  D_{N+N_a-1}(a)
 \end{bmatrix}
 =\mathcal O_a\mathbf v_N^{\rm rc}.
\end{equation}
If $\mathcal O_a$ had rank strictly smaller than $N_a$, the principal
coordinate would factor through an observable quotient of dimension smaller
than $N_a$ and therefore could not have generic minimal recurrence order
$N_a$.  Thus $\mathcal O_a$ is generically nonsingular.  Solving
\eqref{eq:clean-shift-observability} for $\mathbf v_N^{\rm rc}$ and using
\eqref{eq:defect-to-rowcolumn-state} gives
\eqref{eq:defect-observability-elimination}, hence
\eqref{eq:generic-corner-minor-elimination}.
\end{proof}

At special coefficient values the standard-state membership from
Lemma~\ref{lem:standard-state-membership} remains valid, but the observability
step need not.  If $\mathcal O_a$ loses rank, some directions in the standard
row-column state space become invisible from the principal determinant
coordinate.  A particular same-corner minor may still admit a short reduction
after specialization, but that is an additional identity rather than a
consequence of generic observability.

The generic $N_a$-state elimination can collapse dramatically for special
boundary geometries that triangularize before the observability step.  The
following identities are the forms most relevant to the corner expansions in
the main text.

\begin{proposition}[Exceptional one- and two-term corner closures]
\label{prop:exceptional-corner-closures}
For $n,h\ge1$ define
\begin{align}
 \Gamma^+_{n,h}(a)
 &\defeq
 \det\mathcal T(a)
 \bigl[I_{n+h},\,
 I_n\cup\{n+s_a+1,\ldots,n+s_a+h\}\bigr],
 \label{eq:terminal-plus-a}\\
 \Gamma^-_{n,h}(a)
 &\defeq
 \det\mathcal T(a)
 \bigl[I_n\cup\{n+r_a+1,\ldots,n+r_a+h\},\,
 I_{n+h}\bigr].
 \label{eq:terminal-minus-a}
\end{align}
Then
\begin{equation}\label{eq:terminal-chain-a}
 \Gamma^+_{n,h}(a)=a_{s_a}^{\,h}D_n(a),
 \qquad
 \Gamma^-_{n,h}(a)=a_{-r_a}^{\,h}D_n(a).
\end{equation}
If $r_a=s_a$, the aligned boundary minor
\begin{equation}\label{eq:aligned-minor-a}
 \Gamma_n^{\rm al}(a)
 \defeq
 \det\mathcal T(a)
 \bigl[I_n\cup\{n+r_a\},\,I_n\cup\{n+r_a\}\bigr]
\end{equation}
satisfies
\begin{equation}\label{eq:aligned-two-term-a}
 \Gamma_n^{\rm al}(a)
 =a_0D_n(a)-a_{-r_a}a_{s_a}D_{n-1}(a),
 \qquad n\ge1.
\end{equation}
\end{proposition}

\begin{proof}
For \eqref{eq:terminal-plus-a}, order the rows as $I_n$ followed by
$n+1,\ldots,n+h$ and the columns as $I_n$ followed by
$n+s_a+1,\ldots,n+s_a+h$.  The resulting matrix has block form
\[
 \begin{bmatrix}
  \T_n(a)&0\\
  *&U_h
 \end{bmatrix},
 \qquad
 (U_h)_{ij}\defeq a_{s_a+j-i}.
\]
Entries of $U_h$ above the diagonal vanish because their indices exceed
$s_a$, while every diagonal entry equals $a_{s_a}$.  Thus
$\det U_h=a_{s_a}^h$, giving the first identity in
\eqref{eq:terminal-chain-a}; the second is its transpose-dual counterpart.

When $r_a=s_a$, the matrix in \eqref{eq:aligned-minor-a} has block form
\[
 \begin{bmatrix}
  \T_n(a)&a_{s_a}e_n\\
  a_{-r_a}e_n^T&a_0
 \end{bmatrix}.
\]
Expanding in the last row and column leaves either the full principal section
$\T_n(a)$ or its leading $(n-1)\times(n-1)$ section, which gives
\eqref{eq:aligned-two-term-a}.
\end{proof}

Proposition~\ref{prop:exceptional-corner-closures} describes cases in which
the defect geometry disappears before the full state-space elimination is
needed.  The terminal chains model a corner expansion that follows the
outermost diagonal and returns directly to a principal Toeplitz section; the
balanced aligned geometry leaves only two principal sections.  The
pentadiagonal multiplication identity \eqref{eq:scalar-filter} is simpler
still: because the two rank-one corrections in
\eqref{eq:tri-product-corners} sit exactly at $(1,1)$ and $(n,n)$, selecting
one correction leaves $D_{n-1}$ and selecting both leaves $D_{n-2}$.  Thus
\eqref{eq:scalar-filter} is a concrete multiplication-induced short closure.
The deeper cofactors in \eqref{eq:one-column-cofactor-expansion} are balanced
same-corner minors and fall under
Theorem~\ref{thm:generic-corner-minor-elimination}; the recentered endpoint
$M_{n,n}(a)=D_n(za)$ of Proposition~\ref{prop:cauchy-binet-chain} does not.

The companion recurrence paper goes further and classifies the standard states
that close under its prescribed local row-column Laplace rule.  In the
orientation $r_a\ge s_a$, that classification shows that the terminal-chain
states above, together with the aligned state in the balanced case, exhaust
the nonprincipal local closures for that rule, up to the stated transpose-dual
tie convention \cite{AlekseyevKhomovskyRecurrences2026}.  We do not reproduce
that state-graph theorem here: its natural role is in the construction and
classification of the row-column module $Q_a$.  For the present paper its
consequence is only that the one- and two-term formulas above are exceptional
standard-sector geometries; a general same-corner multiplication minor is
handled by the generic observable-state elimination, whereas a cross-corner
imbalance belongs to a different Laurent profile.

\clearpage
\refstepcounter{section}
\section*{Appendix \thesection: Output-sensitive construction for factorized polynomial cores}
\label{app:factorized-construction}

Section~\ref{sec:algebra} shows that a supplied factorization
$Q=q_1\cdots q_m$ can be processed recursively through the graded product
\eqref{eq:star}.  This appendix makes that observation quantitative for a
specific input model: the factorization is regarded as part of the input, the
factor degrees are bounded, and the output is the expanded scalar canonical
annihilator.  The construction uses only the graded multiplication law of
Theorem~\ref{thm:graded-product} and complementary-sector duality from
Theorem~\ref{thm:sector-duality}; it does not require root extraction or the
construction of the full compound matrix.

This input model differs from the row-column method of the companion paper
\cite{AlekseyevKhomovskyRecurrences2026}, which starts from an arbitrary
$(m_1,m_2)$-banded Toeplitz symbol and constructs an exact
$\binom{m_1+m_2}{m_1}$-state sparse transfer from boundary minors.  The two
constructions are therefore complementary, and no claim of universal practical
superiority is intended.

Within the factorized-input model, \emph{output-sensitive} means that the cost
is compared with the number of coefficients in the recurrence polynomial that
must ultimately be produced.  We count arithmetic operations in the
coefficient field and use soft-$O$ notation $\widetilde O(\cdot)$ to suppress
polylogarithmic factors.  The cost of finding the factorization, coefficient
heights, and bit complexity are not included in the main bound; the final
remark discusses these limitations.

Let
\begin{equation}\label{eq:algorithm-factorization}
 Q\defeq q_1q_2\cdots q_m,
 \qquad
 d_j\defeq\deg q_j,
 \qquad
 D\defeq\sum_{j=1}^m d_j,
\end{equation}
and fix a target exterior degree $S$.  Write
\begin{equation}\label{eq:N-output}
 N\defeq\binom DS=\deg\chi_{Q,S}.
\end{equation}
A \emph{factor tree} $\mathcal T$ is a rooted binary tree in which each node
$v$ represents the product $q_v$ of the factors below it, with
$d_v\defeq\deg q_v$.  At an internal node $v$ with children $u,w$ we therefore
have $q_v=q_uq_w$ and $d_v=d_u+d_w$.

\subsection{Pruning to the target exterior degree}

A degree $k$ at a node of degree $d$ can contribute to exterior degree $S$
at the root if and only if the complementary factors can supply the remaining
$S-k$ roots.  Thus the retained set is
\begin{equation}\label{eq:Iv}
 I_v
 \defeq\left\{
 k\in\mathbb Z:
 \max(0,S-D+d_v)\le k\le\min(d_v,S)
 \right\}.
\end{equation}
Since $\deg\chi_{q_v,k}=\binom{d_v}{k}$, set
\begin{equation}\label{eq:Wv}
 W_v
 \defeq
 \sum_{k\in I_v}\binom{d_v}{k}.
\end{equation}
We call $W_v$ the \emph{profile volume} at $v$: up to one constant term per
polynomial, it is the total output size of the retained recurrence
polynomials at that node.  At the root, $I_{\rm root}=\{S\}$ and hence
$W_{\rm root}=N$.

\begin{proposition}[Pruned factor-tree construction]\label{prop:pruned-tree}
Suppose the required profiles at the leaves of a factor tree are available.
For each internal node $v$ with children $u,w$, compute only the polynomials
$\chi_{q_v,k}$ with $k\in I_v$, using
\begin{equation}\label{eq:pruned-star}
 \chi_{q_v,k}
 =
 \prod_{\substack{i+j=k\\i\in I_u,\ j\in I_w}}
 \left(\chi_{q_u,i}\boxtimes\chi_{q_w,j}\right).
\end{equation}
Then the polynomial produced at the root is $\chi_{Q,S}$.
\end{proposition}

\begin{proof}
Theorem~\ref{thm:graded-product} gives \eqref{eq:pruned-star} before pruning.
It remains only to observe that no omitted child degree can contribute to a
retained parent degree.  Indeed, if $i+j=k$ with $0\le j\le d_w$ and the
parent degree $k$ extends to total degree $S$, then $i$ extends to $S$ using
the degree available in $q_w$ together with the factors outside $v$.
Therefore $i\in I_u$; similarly $j\in I_w$.  Iterating from the leaves to the
root proves the claim.
\end{proof}

\paragraph{Using complementary-sector duality.}
Theorem~\ref{thm:sector-duality} allows a duality-aware implementation to
construct independently only one member of each pair $k,d_v-k$.  Define
\begin{equation}\label{eq:Jv}
 J_v\defeq
 \left\{\min(k,d_v-k):k\in I_v\right\}
\end{equation}
and
\begin{equation}\label{eq:Wv-half}
 \widehat W_v\defeq
 \sum_{j\in J_v}\binom{d_v}{j}.
\end{equation}
The polynomials with indices outside $J_v$ are recovered from their partners
by the scaled reciprocal transformation~\eqref{eq:sector-duality}.  Such a
reconstruction is linear in the degree of the polynomial, so it does not
alter the soft-linear complexity bounds below.  It does, however, avoid an
independent composed-product construction for the complementary component and
reduces profile storage.  At the root, one may similarly replace a target
$S>D/2$ by $D-S$, compute $\chi_{Q,D-S}$, and recover $\chi_{Q,S}$ from
Theorem~\ref{thm:sector-duality}.

The quantity $W_v$ is the natural output-size measure at node $v$.  By
Vandermonde's identity,
\begin{equation}\label{eq:Wv-le-N}
 N
 =\sum_{k\in I_v}
   \binom{d_v}{k}\binom{D-d_v}{S-k}
 \ge W_v,
\end{equation}
since every complementary binomial coefficient occurring in the retained
range is at least one.

For later estimates, write
\begin{equation}\label{eq:WDS}
 W_{D,S}(d)
 \defeq
 \sum_{k=\max(0,S-D+d)}^{\min(d,S)}\binom dk.
\end{equation}
Thus $W_v=W_{D,S}(d_v)$.  The following symmetry will be useful.
\begin{lemma}[Complement symmetry]\label{lem:W-symmetry}
For $0\le d\le D$,
\begin{equation}\label{eq:W-symmetry}
 W_{D,S}(d)=W_{D,D-S}(d).
\end{equation}
\end{lemma}

\begin{proof}
Replace a retained $k$-subset inside the node by its complement in the
$d$-element node.  The new local degree is $d-k$, and extendability to a
global $S$-subset becomes extendability to a global $(D-S)$-subset.
Since $\binom dk=\binom d{d-k}$, the volumes agree.
\end{proof}

Set
\begin{equation}\label{eq:S-min}
 S_{\min}\defeq\min(S,D-S).
\end{equation}
By Lemma~\ref{lem:W-symmetry}, all volume estimates may be proved under
$S_{\min}\le D/2$.

\begin{lemma}[Local volume estimate]\label{lem:local-volume}
Assume $2\le S_{\min}\le D/2$ and $d\le D/2$.  Then
\begin{equation}\label{eq:local-volume}
 W_{D,S}(d)
 \le
 1+d+2N\left(\frac dD\right)^2.
\end{equation}
\end{lemma}

\begin{proof}
By complement symmetry we may replace $S$ by $S_{\min}$.  Since $d\le D/2$ and
$S_{\min}\le D/2$, the lower endpoint in \eqref{eq:Iv} is zero, and therefore
\[
 W_{D,S_{\min}}(d)=\sum_{k=0}^{\min(d,S_{\min})}\binom dk.
\]
For every $k\ge0$,
\begin{equation}\label{eq:binom-ratio}
 \frac{\binom dk}{\binom Dk}
 =\prod_{j=0}^{k-1}\frac{d-j}{D-j}
 \le\left(\frac dD\right)^k.
\end{equation}
Moreover $k\le S_{\min}\le D/2$ implies
$\binom Dk\le\binom D{S_{\min}}=N$.  With $x\defeq d/D\le1/2$, the terms $k\ge2$ therefore
satisfy
\[
 \sum_{k=2}^{\min(d,S_{\min})}\binom dk
 \le
 N\sum_{k=2}^{\infty}x^k
 =N\frac{x^2}{1-x}
 \le2Nx^2.
\]
Adding the terms $k=0,1$ proves \eqref{eq:local-volume}.
\end{proof}

\subsection{Balanced trees}

A factor tree is called $\theta$-balanced, where $1/2\le\theta<1$, when every internal
node of degree $d$ has children of degrees $d_1,d_2$ satisfying
\begin{equation}\label{eq:theta-balanced}
 d_1+d_2=d,
 \qquad
 \max(d_1,d_2)\le\theta d.
\end{equation}
The constant $\theta$ is independent of $D$.

\begin{theorem}[Linear total profile volume]\label{thm:balanced-volume}
Let $\mathcal T$ be a $\theta$-balanced factor tree of total degree $D$, and
let $1\le S\le D-1$.  If
\begin{equation}\label{eq:nonedge-condition}
 2\le\min(S,D-S),
\end{equation}
then
\begin{equation}\label{eq:balanced-volume}
 \sum_{v\in\mathcal T}W_v=O_\theta(N),
 \qquad N=\binom DS.
\end{equation}
\end{theorem}

\begin{proof}
Write
\[
 \Lambda\defeq\theta^2+(1-\theta)^2<1.
\]
At every internal split,
\[
 d_1^2+d_2^2\le\Lambda d^2.
\]
Summing by generations gives
\begin{equation}\label{eq:squares-tree}
 \sum_{v\in\mathcal T}d_v^2=O_\theta(D^2).
\end{equation}
The height of the tree is $O_\theta(\log D)$, while the node degrees at each
level sum to at most $D$.  Hence
\begin{equation}\label{eq:degrees-tree}
 \sum_{v\in\mathcal T}d_v=O_\theta(D\log D).
\end{equation}
The number of nodes is $O(D)$ because every leaf has positive integer degree.

There are only $O_\theta(1)$ nodes with $d_v>D/2$: at most one can occur at
any depth, and repeated descent along such a branch decreases its degree by
at least the fixed factor $\theta$.  Each of these nodes contributes at most
$N$ by \eqref{eq:Wv-le-N}.

For all remaining nodes apply Lemma~\ref{lem:local-volume} and sum:
\[
 \sum_{d_v\le D/2}W_v
 \le
 O(D)+O_\theta(D\log D)
 +\frac{2N}{D^2}O_\theta(D^2).
\]
Under \eqref{eq:nonedge-condition}, complement symmetry allows us to assume
$2\le S\le D/2$, and then
\[
 N=\binom DS\ge\binom D2=\Theta(D^2).
\]
Thus $D\log D=O(N)$, proving \eqref{eq:balanced-volume}.
\end{proof}

A \emph{complete dyadic tree}, meaning a tree split exactly in half at every
level, gives a useful explicit constant.

\begin{corollary}[Complete dyadic tree]\label{cor:dyadic-volume}
Suppose $D=2^h$ and the degree tree is split exactly in half at every level,
down to degree one.  If $2\le S\le D-2$, then
\begin{equation}\label{eq:five-N}
 \sum_{v\in\mathcal T}W_v<5\binom DS.
\end{equation}
\end{corollary}

\begin{proof}
By complement symmetry assume $2\le S\le D/2$, where $N\defeq\binom DS$.
At level $\ell=1,\ldots,h$ there are $2^\ell$ nodes of degree
$d_\ell\defeq D/2^\ell$.  Separate the local degrees $k=0,1$ from $k\ge2$.
The first two contribute
\begin{equation}\label{eq:k01-volume}
 \sum_{\ell=1}^{h}2^\ell(1+d_\ell)
 =(2D-2)+hD
 \le\frac73\binom D2
 \le\frac73N.
\end{equation}
For $k\ge2$, \eqref{eq:binom-ratio} gives
\[
 2^\ell\binom{D/2^\ell}{k}
 \le
 2^{-\ell(k-1)}\binom Dk.
\]
Therefore
\begin{align*}
 \sum_{k=2}^{S}\sum_{\ell=1}^{h}
 2^\ell\binom{D/2^\ell}{k}
 &\le
 \sum_{k=2}^{S}
 \frac{\binom Dk}{2^{k-1}-1}\\
 &\le
 N\sum_{j=1}^{\infty}\frac1{2^j-1}
 <\frac53N.
\end{align*}
The last inequality follows from
$1/(2^1-1)=1$ and, for $j\ge2$,
$2^j-1\ge3\cdot2^{j-2}$.  Adding the root contribution $N$ yields
\eqref{eq:five-N}.
\end{proof}

\subsection{The two edge degrees}

The logarithmic overhead of the tree algorithm at $S=1$ and $S=D-1$ is
real.  In the complete dyadic tree,
\begin{equation}\label{eq:edge-tree-volume}
 \sum_{v\in\mathcal T}W_v
 =D(\log_2D+3)-2
 \qquad(S=1),
\end{equation}
and the same holds for $S=D-1$ by complement symmetry.  Nevertheless these
two degrees should not be computed by the tree algorithm: they have direct
linear formulas.

\begin{proposition}[Edge-degree shortcuts]\label{prop:edge-shortcuts}
Let
\[
 q(z)\defeq q_0+q_1z+\cdots+q_dz^d,
 \qquad q_0q_d\ne0,
\]
with the reciprocal polynomial $q^\vee(z)\defeq z^dq(1/z)$.  Then
\begin{align}
 \chi_{q,1}(t)
 &=(-1)^d q_d^{\,d-1}
   q\!\left(-\frac{t}{q_d}\right),
 \label{eq:k1-shortcut}\\
 \chi_{q,d-1}(t)
 &=(-1)^d q_0^{\,d-1}
   q^\vee\!\left(-\frac{t}{q_0}\right).
 \label{eq:kd1-shortcut}
\end{align}
Thus both polynomials can be constructed by coefficient rescaling in
$O(d)$ field operations.
\end{proposition}

\begin{proof}
Writing $\rho_1,\ldots,\rho_d$ for the roots of $q$, the roots of
$\chi_{q,1}$ are $-q_d\rho_i$, and \eqref{eq:k1-shortcut} follows by
factoring $q(-t/q_d)$.  The roots of $\chi_{q,d-1}$ are
$-q_0/\rho_i$, which are the degree-one modes of the reciprocal polynomial
$q^\vee$; applying \eqref{eq:k1-shortcut} to $q^\vee$ gives
\eqref{eq:kd1-shortcut}.
\end{proof}

\subsection{Arithmetic complexity}

We now combine the volume estimate with fast composed products.  Let $\mathsf M(\ell)$ denote a polynomial-multiplication cost bound.  In
characteristic zero, the algorithm of Bostan, Flajolet, Salvy, and Schost
computes the composed product of monic polynomials of degrees $m,n$ in
$O(\mathsf M(mn))$ field operations
\cite{BostanFlajoletSalvySchost2006}.  At an internal node $v$ and fixed
retained degree $k$, the composed-product factors in
\eqref{eq:pruned-star} have degrees
\[
 \binom{d_u}{i}\binom{d_w}{k-i},
\]
whose sum is $\binom{d_v}{k}$ by Vandermonde.  Fast multiplication of these
factors therefore gives total cost
\begin{equation}\label{eq:node-arithmetic}
 \widetilde O\!\left(\binom{d_v}{k}\right)
\end{equation}
for that $k$, and hence $\widetilde O(W_v)$ for the node.

\begin{lemma}[Balanced trees from bounded-degree factors]\label{lem:bounded-degree-tree}
Suppose every given factor $q_j$ in \eqref{eq:algorithm-factorization} has
$1\le d_j\le d_0$, where $d_0$ is fixed.  The factors can be grouped into a
factor tree whose nonterminal splits are $2/3$-balanced and whose terminal
groups have total degree less than $3d_0$.
\end{lemma}

\begin{proof}
At a node of total degree $d\ge3d_0$, add factors to one child until its total
degree first reaches $d/3$.  The overshoot is at most $d_0\le d/3$, so that
child has degree at most $2d/3$; the complementary child also has degree
between $d/3$ and $2d/3$.  Repeat recursively.
\end{proof}

\begin{theorem}[Output-sensitive recurrence construction]\label{thm:output-sensitive}
Assume that the factorization \eqref{eq:algorithm-factorization} is given
over a characteristic-zero field and that the factor degrees are bounded by
a constant $d_0$.  For every $1\le S\le D-1$, the expanded canonical
annihilator $\chi_{Q,S}$ of degree $N\defeq\binom DS$ can be computed,
without extracting roots and without constructing an $N$-dimensional
auxiliary matrix, in
\begin{equation}\label{eq:output-sensitive}
 \widetilde O(N)
\end{equation}
field operations.
\end{theorem}

\begin{proof}
Use Lemma~\ref{lem:bounded-degree-tree}.  By
Theorem~\ref{thm:sector-duality}, a target $S>D/2$ may first be replaced by
$D-S$ and recovered at the end in $O(N)$ field operations, so the tree itself
may work with the smaller complementary exterior degree.  Each terminal group
has bounded total degree, so all independently needed profile polynomials can
be constructed in $O_{d_0}(1)$ field operations; over all terminal groups
this contributes $O_{d_0}(D)$.  If $2\le\min(S,D-S)$,
Theorem~\ref{thm:balanced-volume} and \eqref{eq:node-arithmetic} give total
cost $\widetilde O(N)$.  Complementary local components may be reconstructed
on demand by \eqref{eq:sector-duality}, again within this bound.  If $S=1$
or $S=D-1$, use Proposition~\ref{prop:edge-shortcuts}; here $N=D$, so the
cost is $O(N)$.  This proves the claim.
\end{proof}

\paragraph{A size illustration.}
Table~\ref{tab:profile-volume} gives exact internal profile-volume counts for a
perfectly balanced tree whose leaves are quadratic factors and whose target is
the central exterior degree $S=D/2$.  The leaf profiles are omitted because
their construction has bounded cost.  The third column gives the original
volume $\sum W_v$, while the fourth counts only components that need to be
constructed independently when complementary sectors are recovered by
Theorem~\ref{thm:sector-duality}.  The root output is present in both columns,
so the overall ratios rapidly approach $1$; the saving concerns the non-root
intermediate profile data.

\begin{table}[ht]
\centering
\caption{Internal profile volume for balanced products of quadratic factors,
with and without complementary-sector reconstruction.}
\label{tab:profile-volume}
\begin{tabular}{rrrrr}
\toprule
$D$ & $N=\binom{D}{D/2}$ & $\sum W_v$ & duality-aware volume & ratio to $N$ \\
\midrule
$8$  & $70$        & $102$       & $92$        & $1.3143$ \\
$16$ & $12{,}870$  & $13{,}446$  & $13{,}240$  & $1.02875$ \\
$32$ & $601{,}080{,}390$ & $601{,}212{,}614$ & $601{,}159{,}536$ & $1.00013$ \\
\bottomrule
\end{tabular}
\end{table}

\begin{remark}[What the speedup does and does not say]
Theorem~\ref{thm:output-sensitive} assumes that a useful factorization is
already available over the coefficient field.  Finding such a factorization
may dominate the computation.  Over $\mathbb Q$, coefficient heights can also
grow rapidly, so the field-operation count does not by itself control bit
complexity.  For rational input, a natural exact implementation reduces the
profile arithmetic modulo several primes that avoid the input denominators and
then reconstructs the rational output coefficients by Chinese remaindering
and rational reconstruction; this standard modular strategy limits
intermediate coefficient swell \cite[Chapter~5]{vonZurGathenGerhard2013}.
The number and size of the required primes, however, depend on height bounds
for the output coefficients.  Establishing such bounds for the present
recurrence polynomials lies outside the scope of this paper.

Also, if all roots of $Q$ are already available in the working field, direct
enumeration of the $N$ Widom modes followed by a fast product tree is itself
quasi-linear in $N$.  The factorized algorithm is most useful as a root-free
construction over the original field for structured symbols whose factors
have small degrees or reusable graded profiles.  Complementary-sector duality
reduces the independently constructed local profile to essentially one half
at large local degrees, but it does not change the overall
$\widetilde O(N)$ order because the root polynomial itself already has degree
$N$.  Table~\ref{tab:profile-volume} is therefore a size illustration, not a
software benchmark; implementation
comparisons depend on coefficient heights, arithmetic representation, and the
generic algorithms selected by a computer algebra system.
\end{remark}

\end{document}